\documentclass[reqno]{amsart}

\usepackage{microtype}
\usepackage{hyperref}
         
\usepackage{amsmath}             
\usepackage{amsfonts}             
\usepackage{amsthm}               
\usepackage{amsbsy}
\usepackage{amssymb}
\usepackage{amsthm}
\usepackage[initials]{amsrefs}

\usepackage{enumerate}

\newtheorem{thm}{Theorem}[section]
\newtheorem{lem}[thm]{Lemma}
\newtheorem{prop}[thm]{Proposition}
\newtheorem{cor}[thm]{Corollary}

\theoremstyle{definition}
\newtheorem{defn}[thm]{Definition}

\newtheorem{rem}[thm]{Remark}
\newtheorem{exam}[thm]{Example}

\newcommand{\bC}{{\mathbb{C}}}
\newcommand{\bN}{{\mathbb{N}}}
\newcommand{\bR}{{\mathbb{R}}}

\newcommand{\bQ}{{\mathbb{Q}}}

\newcommand{\B}{{\mathcal{B}}}

\renewcommand{\H}{{\mathcal{H}}}
\newcommand{\I}{{\mathcal{I}}}

\newcommand{\M}{{\mathcal{M}}}
\newcommand{\N}{{\mathcal{N}}}
\renewcommand{\O}{{\mathcal{O}}}

\newcommand{\U}{{\mathcal{U}}}

\renewcommand{\phi}{\varphi}

\newcommand{\fA}{{\mathfrak{A}}}

\newcommand{\fM}{{\mathfrak{M}}}

\newcommand{\qand}{\quad\text{and}\quad}
\newcommand{\qqand}{\qquad\text{and}\qquad}

\newcommand{\tr}{\mathrm{Tr}}
\newcommand{\diag}{\mathrm{diag}}

\newcommand{\conv}{\mathrm{conv}}
\newcommand{\cconv}{\overline{\mathrm{conv}}}
\newcommand{\dist}{\mathrm{dist}}

\begin{document}

\nocite{*}

\title[Closed Convex Hulls of Unitary Orbits]{Closed Convex Hulls of Unitary Orbits in C$^*$-Algebras of Real Rank Zero}

\author{Paul Skoufranis}
\address{Department of Mathematics, Texas A\&M University, College Station, Texas, USA, 77843-3368}

\email{pskoufra@math.tamu.edu}

\date{\today}

\begin{abstract}
In this paper, we study closed convex hulls of unitary orbits in various C$^*$-algebras.  For unital C$^*$-algebras with real rank zero and a faithful tracial state determining equivalence of projections, a notion of majorization describes the closed convex hulls of unitary orbits for self-adjoint operators.  Other notions of majorization are examined in these C$^*$-algebras.  Combining these ideas with the Dixmier property, we demonstrate unital, infinite dimensional C$^*$-algebras of real rank zero and strict comparison of projections with respect to a faithful tracial state must be simple and have a unique tracial state.  Also, closed convex hulls of unitary orbits of self-adjoint operators are fully described in unital, simple, purely infinite C$^*$-algebras.
\end{abstract}

\maketitle

\section{Introduction}

Unitary orbits of operators are important objects that provide significant information about operators.  In the infinite dimensional setting, the norm closure of the unitary orbits must be taken as unitary groups are no longer compact.  For all intents and purposes, two operators that are approximately unitarily equivalent (that is, have the same closed unitary orbits) may be treated as the same operator inside a C$^*$-algebra and the question of when two (normal) operators are approximately unitarily equivalent has been studied in a variety of contexts (e.g. \cites{D1995, S2007}).

When two operators are not approximately unitarily equivalent, it is interesting to ask, ``How far are the operators from being approximately unitarily equivalent?"  This question is quantified by describing the distance between the operators' unitary orbits and has a long history.  For self-adjoint matrices $S$ and $T$ with eigenvalues $\{\mu_k\}^n_{k=1}$ and $\{\lambda_k\}^n_{k=1}$ respectively, the distance between the unitary orbits of $S$ and $T$ was computed in \cite{W1912} to be the optimal matching distance
\[
\min_{\sigma \in S_n} \max\{ |\lambda_k - \mu_{\sigma(k)}| \, \mid \, k \in \{1,\ldots, n\}\}
\]
where $S_n$ is the permutation group on $\{1,\ldots, n\}$.  However, if $S$ and $T$ are normal matrices, the distance between the unitary orbits of $S$ and $T$  need not equal the optimal matching distance (see \cite{H1992}).  For bounded normal operators on Hilbert space, results have been obtained analogous to the known matricial results (e.g. \cites{AD1984, D1988}).  This question has been active in other C$^*$-algebras (e.g. \cites{C2013, D1984, D1986, HN1989, HL2015, JST2015, S2013-1, ST1992}) where the most recent work has made use of $K$-theoretic properties and ideas.

Another important concept is that of majorization for self-adjoint matrices.  A notion of majorization for real-valued functions in $L_1[0,1]$ was first developed in \cite{HLP1929} by Hardy, Littlewood, and P\'{o}lya using non-increasing rearrangements and this notion has been widely studied (e.g. \cites{C1974, C1976, HLP1952, S1985}).  When applied to self-adjoint matrices through their eigenvalues, a fascinating concept is obtained.  Majorization of self-adjoint matrices has been thoroughly analyzed (e.g. \cites{A1989, B1946, KW2010, MOA2010, M1963, M1964, T1971}) and has relations to a wide range of problems in linear algebra, such as  classical theorem of Schur and Horn characterizing the possible diagonal $n$-tuples of a self-adjoint matrix based on its eigenvalues (see \cites{S1923, H1954}) and applications to generalized numerical ranges of matrices (see \cites{GS1977, P1980}).

Majorization has an immediate analogue in II$_1$ factors by replacing eigenvalues with spectral distributions.  By using the notion of majorization in \cite{K1983} (also see \cites{AM2008, AK2006, F1982, FK1986, H1987, K1984, K1985, P1985}) via eigenvalue functions (also known as spectral scales) of self-adjoint operators in II$_1$ factors, several analogues of matricial results have been obtained.  For example, an analogue of the Schur-Horn Theorem for II$_1$ factors was first postulated in \cite{AK2006} and proved by Ravichandran in \cite{R2012} (also see \cites{T1977, KS2015} for a generalization to non-self-adjoint operators using singular values, and \cite{MR2014} for a multivariate version) and analogues of generalized numerical ranges were developed in \cite{DS2015}.

The notion of majorization of self-adjoint operators in both matrix algebras and II$_1$ factors as a deep connection with unitary orbits.  Indeed, given two self-adjoint operators $S$ and $T$, it was shown for matrix algebras in \cite{A1989} and II$_1$ factors in \cites{K1983, K1984, K1985} that $T$ majorizes $S$ if and only if $S$ is in the (norm) closure of the convex hull of the unitary orbit of $T$, denoted $\cconv(\U(T))$.  Consequently, the question of whether $T$ majorizes $S$ is a question of whether $S$ can be obtained by `averaging' copies of $T$. 

Analysis of the closure of convex hulls of unitary orbits has yielded some interesting results.  For example, the Dixmier property for a C$^*$-algebra (\cite{D1969}) asks that the centre of the C$^*$-algebra interests every such orbit.  By \cite{R1983}, one need only consider self-adjoint operators to verify the Dixmier property and \cite{HZ1984} (also see \cite{R1982}) shows that a unital C$^*$-algebra $\fA$ has the Dixmier property if and only if $\fA$ is simple and has at most one faithful tracial state.

The goal of this paper is to describe the closure of convex hulls of unitary orbits of self-adjoint operators in various C$^*$-algebras.  Taking inspiration from von Neumann algebra theory, we will focus on C$^*$-algebras that behave like type III and type II$_1$ factors.  In particular, unital, simple, purely infinite C$^*$-algebras are our analogues of type III factors and unital C$^*$-algebras with real rank zero and a faithful tracial state determining equivalence of projections are our analogues of type II$_1$ factors.  In addition to this introduction, this paper contains five sections with results and their importance summarized below.

Section \ref{sec:Preliminaries} develops and extends the necessarily preliminary results on majorization of self-adjoint operators in matrix algebras and II$_1$ factors.  In particular, the notion of eigenvalue functions is adapted from II$_1$ factors to C$^*$-algebras with faithful tracial states by replacing spectral distributions with dimension functions (see Definition \ref{defn:eigenvalues-and-singular-values}).  The properties of eigenvalue functions are immediately transferred to this setting (see Theorem \ref{thm:properties-of-eigenvalue-functions}).  

Section \ref{sec:Simplicity} analyzes whether there are scalars in convex hulls of unitary orbits in C$^*$-algebras with faithful tracial states.  Notice if $\fA$ is a unital C$^*$-algebra with a faithful tracial state $\tau$ and $T \in \fA$, then $\tau(S) = \tau(T)$ for all $S \in \cconv(\U(T))$.  Consequently $\cconv(\U(T)) \cap \{\bC I_\fA\}$ is either empty or $\{\tau(T) I_\fA\}$.  Using an averaging process along with manipulations of projections, it is demonstrated in Theorem \ref{thm:scalars-in-convex-hulls} that if $\fA$ is a unital, infinite dimensional C$^*$-algebra with real rank zero and strict comparison of projections with respect to a faithful tracial state $\tau$, then $\tau(T) \in \cconv(\U(T))$ for all $T \in \fA$.  Combined with the Dixmier property, this implies $\fA$ must be simple and $\tau$ must be the unique faithful tracial state on $\fA$.  We note \cite{M2000} has also investigated the ability of faithful tracial states to imply simplicity of C$^*$-algebras.

Section \ref{sec:Convex-Hulls-Of-Unitary-Orbits} analyzes $\cconv(\U(T))$ for self-adjoint $T$ in unital C$^*$-algebras $\fA$ that have real rank zero and a faithful tracial state $\tau$ with the property that if $P, Q \in \fA$ are projections, then $\tau(P) \leq \tau(Q)$ if and only if $P$ is Murray-von Neumann equivalent to a subprojection of $Q$.  In particular, Theorem \ref{thm:classification-of-convex-hull} shows for such C$^*$-algebras that $S \in \cconv(\U(T))$ if and only if $T$ majorizes $S$ with respect to $\tau$.  Although the assumptions on $\fA$ are restrictive in the classification theory world, they do apply to several C$^*$-algebras such as UHF C$^*$-algebras, the Bunce-Deddens C$^*$-algebras, irrational rotations algebras, and many others.

Trying to generalize Theorem \ref{thm:classification-of-convex-hull} to other C$^*$-algebras may be a difficult task.  Indeed, it is the case that there are self-adjoint operators with the same eigenvalue functions that are not approximately unitarily equivalent when the assumption `$\tau(P) = \tau(Q)$ implies $P$ and $Q$ are equivalent' is removed.  In addition, the question of characterizing $\cconv(\U(T))$ appear very complicated if $\fA$ has more than one tracial state as, by above discussions, $\cconv(\U(T))\cap \{\bC I_\fA\} = \emptyset$.

Section \ref{sec:Other-Majorizations} is devoted to investigating other closed orbits and notions of majorization of operators in the same context as Section \ref{sec:Convex-Hulls-Of-Unitary-Orbits}.  We begin by using eigenvalue functions to re-derive the main result of \cite{ST1992}, which computes the distance between unitary orbits of self-adjoint operators via an analogue of the optimal matching distance (see Theorem \ref{thm:distance-between-unitary-orbits}).  In addition, an analogue of singular value decomposition of matrices is obtained (see Proposition \ref{prop:closed-two-sided-unitary-orbit}).  Furthermore, descriptions of when one operator's eigenvalue (singular value) function dominants another operator's eigenvalue (respectively singular value) function and when one operator (absolutely) submajorizes another operator are described.

Section \ref{sec:Purely-Infinite} concludes the paper by describing $\cconv(\U(T))$ for self-adjoint operators $T$ in unital, simple, purely infinite C$^*$-algebras.  In particular, $\cconv(\U(T))$ is precisely all self-adjoint operators $S$ such that the spectrum of $S$ is contained in the convex hull of the spectrum of $T$ (Theorem \ref{thm:convex-hull-of-unitary-orbit-in-uspi}).

\section{Preliminaries}
\label{sec:Preliminaries}

In this section, we develop the preliminaries necessary for the remainder of the paper.  In particular, after the following definitions, we will extend the notion and properties of eigenvalue functions to C$^*$-algebras with faithful tracial states.
\begin{defn}
For a unital C$^*$-algebra $\fA$ and an element $T \in \fA$, the \emph{unitary orbit} of $T$ is
\[
\U(T) := \{U^*TU \, \mid \, U \text{ a unitary in }\fA\}.
\]
The \emph{closed unitary orbit} of $T \in \fA$ is $\O(T)  := \overline{\U(T)}$, the norm closure of $\U(T)$.
The convex hull of $\U(T)$ will be denoted by $\conv(\U(T))$ and its norm closure by $\cconv(\U(T))$.
\end{defn}

One main component of this paper is the generalization of the following notions from tracial von Neumann algebras to tracial C$^*$-algebras.  The origins of the following definition may be traced back to \cite{MN1936}.
\begin{defn}[\cites{MN1936, F1982, FK1986}]
\label{defn:eigenvalues-and-singular-values-vN}
Let $\fM$ be a von Neumann algebra a tracial state $\tau$.  
\begin{enumerate}
\item For a self-adjoint operator $T \in \fM$, the \emph{eigenvalue function of} $T$ \emph{associated with} $\tau$, denoted $\lambda^\tau_T$, is defined for $s \in [0,1)$ by
\[
\lambda^\tau_T(s) := \inf\{t \in \bR \, \mid \, m_T((t,\infty)) \leq s\}
\]
where $m_T$ is the spectral distribution of $T$ with respect to $\tau$.
\item For an arbitrary $T \in \fM$, the \emph{singular value function of} $T$ \emph{associated with} $\tau$, denoted $\mu^\tau_T$, is defined for $s \in [0,1)$ by
\[
\mu^\tau_T(s) := \lambda^\tau_{|T|}(s).
\]
\end{enumerate}
\end{defn}

\begin{exam}
Let $T \in \M_n(\bC)$ be self-adjoint with eigenvalues $\{\lambda_k\}^n_{k=1}$ where $\lambda_k \geq \lambda_{k+1}$ for all $k$.  If $\tau$ is the normalized trace on $\M_n(\bC)$, then $\lambda^\tau_T(s) = \lambda_k$ for all $s \in \left[ \frac{k-1}{n}, \frac{k}{n}  \right)$.  Similarly, if $T \in \M_n(\bC)$ has singular values $\{\sigma_k\}^n_{k=1}$ where $\sigma_k \geq \sigma_{k+1}$ for all $k$, then $\mu^\tau_T(s) = \mu_k$ for all $s \in \left[ \frac{k-1}{n}, \frac{k}{n}  \right)$.
\end{exam}

\begin{exam}
\label{exam:non-increasing-rearrangements}
Let $\fM = L_\infty[0,1]$ equipped with the tracial state $\tau$ defined by integrating against the Lebesgue measure $m$.  If $f \in \fM$ is real-valued, then $\lambda^\tau_f(s) = f^*(s)$ where $f^*$ is the non-increasing rearrangement of $f$, which may be defined by
\[
f^*(s) := \inf\{t \in \bR \, \mid \, m(\{x \in [0,1] \, \mid \, f(x) > t \}) \leq s\}.
\]
It can be shown (see Theorem \ref{thm:properties-of-eigenvalue-functions}) that $f^*$ is a non-increasing, right continuous function.  Consequently, if $f$ is non-increasing and right continuous, then $f = f^*$.
\end{exam}

To generalize these notions to C$^*$-algebras with faithful tracial states, we will use the following as a replacement for spectral distributions.
\begin{defn}[\cite{C1978}]
Let $\epsilon > 0$ and let $f_\epsilon$ denote the continuous function on $[0, \infty)$ such that $f_\epsilon(x) = 1$ if $x \in [\epsilon, \infty)$, $f_\epsilon(x) = 0$ if $x \in [0, \frac{\epsilon}{2}]$, and $f_\epsilon(x)$ is linear on $(\frac{\epsilon}{2}, \epsilon)$.

Let $\fA$ be a unital C$^*$-algebra with faithful tracial state $\tau$.  The \emph{dimension function associated with $\tau$}, denoted $d_\tau$, is defined for positive operators $A \in \fA$ by
\[
d_\tau(A) := \lim_{\epsilon \to 0} \tau(f_\epsilon(A)).
\]
\end{defn}

\begin{defn}
\label{defn:eigenvalues-and-singular-values}
Let $\fA$ be a unital C$^*$-algebra with a faithful tracial state $\tau$.  
\begin{enumerate}
\item For a self-adjoint operator $T \in \fA$, the \emph{eigenvalue function of} $T$ \emph{associated with} $\tau$, denoted $\lambda^\tau_T$, is defined for $s \in [0,1)$ by
\[
\lambda^\tau_T(s) := \inf\{t \in \bR \, \mid \, d_\tau((T - t I_\fA)_+) \leq s\}
\]
where $(T - tI_\fA)_+$ denotes the positive part of $T - t I_\fA$.
\item For an arbitrary $T \in \fA$, the \emph{singular value function of} $T$ \emph{associated with} $\tau$, denoted $\mu^\tau_T$, is defined for $s \in [0,1)$ by
\[
\mu^\tau_T(s) := \lambda^\tau_{|T|}(s).
\]
\end{enumerate}
\end{defn}

Lemma \ref{lem:eigenvalue-functions-come-from-vN-algebras} will demonstrate that Definitions \ref{defn:eigenvalues-and-singular-values-vN} and \ref{defn:eigenvalues-and-singular-values} agree when $\fA$ is a von Neumann algebra.

\begin{exam}
\label{exam:finite-spectrum-eigenvalue-functions}
Let $\fA$ be a unital C$^*$-algebra with a faithful tracial state $\tau$.  Let $\{\lambda_k\}^n_{k=1} \subseteq \bR$ be such that $\lambda_k \geq \lambda_{k+1}$ for all $k$ and let $\{P_k\}^n_{k=1} \subseteq \fA$ be a collection of pairwise orthogonal projections such that $\sum^n_{k=1} P_k = I_\fA$.  For each $k \in \{0, 1,\ldots, n\}$, let $s_k = \sum^k_{j=1} \tau(P_j)$.  If $T = \sum^n_{k=1} \lambda_k P_k$, then $\lambda^\tau_T(s) = \lambda_k$ for all $s \in [s_{k-1}, s_k)$.
\end{exam}

\begin{rem}
Part (\ref{part:decreasing-and-continuous}) of Theorem \ref{thm:properties-of-eigenvalue-functions} demonstrates that eigenvalue functions are non-increasing and right continuous.  If $\fM$ is a diffuse von Neumann algebra, it is not difficult to show every non-increasing, right continuous function is the eigenvalue function of some self-adjoint operator in $\fM$.  Example \ref{exam:finite-spectrum-eigenvalue-functions} shows this is not the case for arbitrary C$^*$-algebras as the characteristic function of the set $[0, \alpha)$ is an eigenvalue function of a self-adjoint operator in $\fA$ if and only if $\fA$ has a projection of trace $\alpha$.
\end{rem}

Eigenvalue and singular value functions have several important properties.  Although most (if not all) of these properties can be demonstrated using C$^*$-algebraic techniques, we will appeal to von Neumann algebra theory to shorten the exposition.

For a unital C$^*$-algebra $\fA$ with a faithful tracial state $\tau$, let $\pi_\tau : \fA \to \B(L_2(\fA, \tau))$ be the GNS representation of $\fA$ with respect to $\tau$.  Note $\pi_\tau$ is faithful and $\tau$ is a vector state on $\B(L_2(\fA, \tau))$.  If $\fM$ is the von Neumann algebra generated by $\pi_\tau(\fA)$, specifically $\pi_\tau(\fA)''$, then $\tau$ extends to a tracial state on $\fM$.
\begin{lem}
\label{lem:eigenvalue-functions-come-from-vN-algebras}
Let $\fA$ be a unital C$^*$-algebra with faithful tracial state $\tau$ and let $\fM$ be the von Neumann algebra described above.  If $T \in \fA$ is self-adjoint, then
\[
\lambda^\tau_T(s) = \lambda^\tau_{\pi_\tau(T)}(s)
\]
for all $s \in [0,1)$, where $\lambda^\tau_{\pi_\tau(T)}$ is as defined in Definition \ref{defn:eigenvalues-and-singular-values-vN}.
\end{lem}
\begin{proof}
If $m_{\pi_\tau(T)}$ denotes the spectral distribution of $\pi_\tau(T)$ with respect to $\tau$, we obtain for all $t \in \bR$ that
\begin{align*}
d_\tau((T- tI_\fA)_+) &= \lim_{\epsilon \to 0} \tau(f_\epsilon((T- tI_\fA)_+)) \\
&=\lim_{\epsilon \to 0} \tau(\pi_\tau(f_\epsilon((T- tI_\fA)_+)))  \\
&=\lim_{\epsilon \to 0} \tau(f_\epsilon(\pi_\tau(T- tI_\fA)_+))= m_{\pi_\tau(T)}((t, \infty))
\end{align*}
as $f_\epsilon(\pi_\tau(T- tI_\fA)_+)$ converges in the weak$^*$-topology to the spectral projection of $\pi_\tau(T)$ onto $(t, \infty)$.  The result then follows by definitions.
\end{proof}

Using Lemma \ref{lem:eigenvalue-functions-come-from-vN-algebras}, the known properties of eigenvalue and singular value functions on von Neumann algebras automatically transfer to the tracial C$^*$-algebra setting.
\begin{thm}[see \cites{F1982, FK1986, P1985}]
\label{thm:properties-of-eigenvalue-functions}
Let $\fA$ be a unital C$^*$-algebra with faithful tracial state $\tau$ and let $T, S \in \fA$ be self-adjoint operators.  Then:
\begin{enumerate}
\item The map $s \mapsto \lambda^\tau_T(s)$ is non-increasing and right continuous.  \label{part:decreasing-and-continuous}
\item If $T \geq 0$, $\lim_{s \searrow 0} \lambda^\tau_T(s) = \left\|T\right\|$ and $\lambda^\tau_T(s) \geq 0$ for all $s \in [0,1)$. \label{part:norm-at-zero}
\item If $\sigma(T)$ denotes the spectrum of $T$, then $\lim_{s \nearrow 1} \lambda^\tau_T(s) = \inf\{t \, \mid \, t \in \sigma(T)\}$ and $\lim_{s \searrow 0} \lambda^\tau_T(s) = \sup\{t \, \mid \, t \in \sigma(T)\}$.
\item If $S \leq T$, then $\lambda^\tau_S(s) \leq \lambda^\tau_T(s)$ for all $s \in [0,1)$. \label{part:larger}
\item If $\alpha \geq 0$, then $\lambda^\tau_{\alpha T}(s) = \alpha \lambda^\tau_T(s)$ and $\lambda^\tau_{T + \alpha I_\fA} = \lambda^\tau_T(s) + \alpha$ for all $s \in [0,1)$. \label{part:positive-multiple}
\item $\lambda^\tau_{S + T}(s + t) \leq \lambda^\tau_S(s) + \lambda^\tau_T(t)$ for all $s, t \in [0,1)$ with $s + t < 1$.
\item $|\lambda^\tau_S(s) - \lambda^\tau_T(s)| \leq \left\|S - T\right\|$ for all $s \in [0,1)$. \label{part:difference}
\item $\tau(f(T)) = \int^1_0 f(\lambda^\tau_T(s)) \, ds$ for all continuous functions $f : \bR \to \bR$.\label{part:trace}
\item If $T \geq 0$, then $\lambda^\tau_{V^*TV}(s) \leq \left\|V\right\|^2 \lambda^\tau_{T}(s)$ for all $s \in [0,1)$ and $V \in \fA$. \label{part:conjugation}
\item If $U \in \fA$ is a unitary, then $\lambda^\tau_{U^*TU}(s) = \lambda^\tau_T(s)$ for all $s \in [0,1)$.\label{part:unitary}
\item If $T \geq 0$, $\lambda^\tau_{f(T)}(s) = f(\lambda^\tau_T(s))$ for all $s \in [0,1)$ and all continuous increasing functions $f : [0,\infty) \to \bR$ with $f(0) \geq 0$.
\item If $S, T \geq 0$, then $\int^t_0 f(\lambda^\tau_{S+T}(s)) \, ds \leq \int^t_0 f(\lambda^\tau_S(s) + \lambda^\tau_T(s)) \, ds$ for all $t \in [0,1]$ and all continuous, increasing, convex functions $f : \bR \to \bR$. \label{part:convex-integral}
\item If $S, T \geq 0$, then $\int^t_0 f(\lambda^\tau_{S+T}(s)) \, ds \leq \int^t_0 f(\lambda^\tau_S(s)) + f(\lambda^\tau_T(s)) \, ds$ for all $t \in [0,1]$ and all increasing concave functions $f : \bR \to \bR$ with $f(0) = 0$. \label{part:concave-integral}
\end{enumerate}
\end{thm}

\begin{thm}[see \cites{F1982, FK1986}]
\label{thm:properties-of-singular-value-functions}
Let $\fA$ be a unital C$^*$-algebra with faithful tracial state $\tau$ and let $T, S, R \in \fA$.  Then:
\begin{enumerate}
\item $\mu^\tau_T(s) = \mu^\tau_{|T|}(s) = \mu^\tau_{T^*}(s)$ for all $s \in [0,1)$.
\item $\mu^\tau_{\alpha T}(s) = |\alpha|\mu^\tau_T(s)$ for all $s \in [0,1)$ and $\alpha \in \bC$. \label{part:mu-scalar}
\item $\mu^\tau_{RTS}(s) \leq \left\|R\right\|\left\|S\right\| \mu^\tau_T(s)$ for all $s \in [0,1)$. \label{part:mu-contractive}
\item $\mu^\tau_{ST}(s + t) \leq \mu^\tau_S(s)\mu^\tau_T(t)$ for all $s, t \in [0,1)$ with $s + t < 1$.
\item $\int^t_0 f(\mu^\tau_{S+T}(s)) \, ds \leq \int^t_0 f(\mu^\tau_S(s) + \mu^\tau_T(s)) \, ds$ for all $t \in [0,1]$ and all continuous, increasing, convex functions $f : \bR \to \bR$. \label{part:mu-convex-integral}
\item $\int^t_0 f(\mu^\tau_{S+T}(s)) \, ds \leq \int^t_0 f(\mu^\tau_S(s)) + f(\mu^\tau_T(s)) \, ds$ for all $t \in [0,1]$ and all increasing concave functions $f : \bR \to \bR$ with $f(0) = 0$. \label{part:mu-concave-integral}
\end{enumerate}
\end{thm}

To define a notion of majorization for self-adjoint operators, we recall the following.

\begin{defn}[\cite{HLP1929}]
\label{defn:majorization-functions}
For real-valued functions $f, g \in L_\infty[0,1]$, it is said that $f$ \emph{majorizes} $g$, denoted $g \prec f$, if
\begin{align*}
\int^t_0 g^*(s) \, ds &\leq \int^t_0 f^*(s) \, ds \text{ for all } t \in [0,1] \qand
\int^1_0 g^*(s) \, ds &= \int^1_0 f^*(s) \, ds
\end{align*}
where $f^*$ and $g^*$ are the non-increasing rearrangements of $f$ and $g$ (see Example \ref{exam:non-increasing-rearrangements}).
\end{defn}

The following example provides some intuition for majorization and will be used in various forms later in the paper.
\begin{exam}
\label{exam:averaging}
Let $f \in L_\infty[0,1]$ be a real-valued function and fix $\{0 = s_0 < s_1 < \cdots < s_n = 1\}$.  For $k \in \{1,\ldots, n\}$, let
\[
\alpha_k = \frac{1}{s_k - s_{k-1}} \int^{s_k}_{s_{k-1}} f^*(s) \, ds
\]
and let $g= \sum^n_{k=1} \alpha_k 1_{[s_{k-1}, s_k)}$, where $1_X$ denotes the characteristic function of $X$.  We claim that $g \prec f$. Note $g$ is non-increasing and right continuous so $g^* = g$.  Furthermore, note
\[
\int^{s_k}_0 f^*(s) \, ds = \int^{s_k}_0 g(s) \, ds
\]
for all $k \in \{0, 1, \ldots, n\}$.  

Suppose $t \in [s_{k-1}, s_k]$.  If $g(t) \leq f^*(t)$, then $g(s) \leq f^*(s)$ for all $s \in [s_{k-1}, t]$ as $g$ is constant on $[s_{k-1}, s_k)$ and $f^*$ is non-increasing.  Thus
\[
\int^{t}_0 f^*(s) - g(s) \, ds = \int^{t}_{s_{k-1}} f^*(s) - g(s) \, ds \geq 0.
\]
Otherwise $g(t) > f^*(t)$.  Hence $g(s) > f^*(s)$ for all $s \in [t, s_k)$ as $g$ is constant on $[s_{k-1}, s_k)$ and $f^*$ is non-increasing. Thus
\[
\int^{t}_0 f^*(s) - g(s) \, ds = \int^{t}_{s_{k-1}} f^*(s) - g(s) \, ds \geq \int^{s_k}_{s_{k-1}} f^*(s) - g(s) \, ds = 0.
\]
Hence $g \prec f$ as claimed.
\end{exam}

\begin{defn}
\label{defn:majorization}
Let $\fA$ be a unital C$^*$-algebra with a faithful tracial state $\tau$.  For self-adjoint elements $T, S \in \fA$, it is said that $T$ \emph{majorizes} $S$ \emph{with respect to} $\tau$, denoted $S \prec_\tau T$, if $\lambda^\tau_S \prec \lambda^\tau_T$.
\end{defn}
\begin{rem}
Note by part (\ref{part:decreasing-and-continuous}) of Theorem \ref{thm:properties-of-eigenvalue-functions} that eigenvalue functions are equal to their non-increasing rearrangements.  Therefore, for $S$ and $T$ as in Definition \ref{defn:majorization},  we obtain $S \prec_\tau T$ if and only if
\begin{align*}
\int^t_0 \lambda^\tau_S(s) \, ds &\leq \int^t_0 \lambda^\tau_T(s) \, ds \text{ for all } t \in [0,1] \qand
\int^1_0 \lambda^\tau_S(s) \, ds &= \int^1_0 \lambda^\tau_T(s) \, ds.
\end{align*}
Furthermore, by part (\ref{part:trace}) of Theorem \ref{thm:properties-of-eigenvalue-functions}, the later condition is equivalent to $\tau(S) = \tau(T)$.
\end{rem}

\begin{exam}
\label{exam:majorization-self-adjoint-matrices}
Let $T, S \in \M_n(\bC)$ be self-adjoint with eigenvalues 
\[
\{\lambda_1 \geq \lambda_2 \geq \cdots \geq \lambda_n\} \qand \{\mu_1 \geq \mu_2 \geq \cdots \geq \mu_n\}
\]
respectively.   If $\tau$ is the normalized trace on $\M_n(\bC)$, then $S \prec_\tau T$ if and only if
\begin{align*}
\sum^m_{k=1} \mu_k \leq \sum^m_{k=1} \lambda_k  \text{ for all }m \in \{1,\ldots, n-1\}\qqand \sum^n_{k=1} \mu_k = \sum^n_{k=1} \lambda_k.
\end{align*}
\end{exam}

\begin{rem}
It is a consequence of part (\ref{part:positive-multiple}) of Theorem \ref{thm:properties-of-eigenvalue-functions} that if $S, T \in \fA$ are self-adjoint, then $S \prec_\tau T$ if and only if $S + \alpha I_\fA \prec_\tau T + \alpha I_\fA$ for any $\alpha\in \bR$.  Consequently, it often suffices to consider positive operators when demonstrating results involving majorization.
\end{rem}

There are several equivalent formulations of majorization of self-adjoint operators in tracial von Neuman algebras as the following theorem demonstrates.
\begin{thm}[see \cites{AK2006, AM2008, A1989, H1987, HN1991, K1983, K1984, K1985}]
\label{thm:majorization-in-factors}
Let $\fM$ be a von  Neumann algebra with a faithful tracial state $\tau$.  Let $T, S \in \fM$ be positive operators.  Then the following are equivalent:
\begin{enumerate}
\item $S \prec_\tau T$.
\item $\tau((S- r I_\fM)_+) \leq \tau((T - r I_\fM)_+)$ for all $r > 0$ and $\tau(T) = \tau(S)$.
\item $\tau(f(S)) \leq \tau(f(T))$ for every continuous convex function $f : \bR \to \bR$. 
\end{enumerate}
If $\fM$ is a factor, then for all self-adjoint $S, T \in \fM$, $S \prec_\tau T$ is equivalent to:
\begin{enumerate}
\setcounter{enumi}{3}
\item $S \in \cconv(\U(T))$. \label{part:convex}
\item $S \in \overline{\conv(\U(T))}^{w^*}$. \label{part:weak*}
\item There exists a unital, trace-preserving, positive map $\Phi : \fM \to \fM$ such that $\Phi(T) = S$. \label{part:positive-map}
\item There exists a unital, trace-preserving, completely positive map $\Phi : \fM \to \fM$ such that $\Phi(T) = S$. \label{part:cp-map}
\end{enumerate}
\end{thm}

One goal of this paper is to see to what extent Theorem \ref{thm:majorization-in-factors} generalizes to tracial C$^*$-algebras.  Note Lemma \ref{lem:eigenvalue-functions-come-from-vN-algebras} immediately implies the following.
\begin{cor}
\label{cor:versions-of-majorization-for-positives}
Let $\fA$ be a unital C$^*$-algebra with a faithful tracial state $\tau$.  Let $T, S \in \fA$ be positive operators.  Then the following are equivalent:
\begin{enumerate}
\item $S \prec_\tau T$.
\item $\tau((S- r I_\fM)_+) \leq \tau((T - r I_\fM)_+)$ for all $r > 0$ and $\tau(T) = \tau(S)$.
\item $\tau(f(S)) \leq \tau(f(T))$ for every continuous convex function $f : \bR \to \bR$. 
\end{enumerate}
\end{cor}

For the remaining equivalences in Theorem \ref{thm:majorization-in-factors}, note part (\ref{part:weak*}) does not make sense in an arbitrary C$^*$-algebra.  We will mainly focus on part (\ref{part:convex}) of Theorem \ref{thm:majorization-in-factors} to which we have the following preliminary result.
\begin{lem}
\label{lem:easy-direction-for-convex-hulls}
Let $\fA$ be a unital C$^*$-algebra with a faithful tracial state $\tau$ and let $T \in \fA$ be self-adjoint.  Then
\begin{enumerate}
\item If $\lambda \in \bR$, then $\lambda I_\fA \prec_\tau T$ if and only if $\lambda = \tau(T)$
\item If $S \in \cconv(\U(T))$, then $S = S^*$ and $S \prec_\tau T$.
\end{enumerate}
\end{lem}
\begin{proof}
The first claim follows from Example \ref{exam:averaging} and part (\ref{part:trace}) of Theorem \ref{thm:properties-of-eigenvalue-functions}.  

For the second claim, suppose $\{U_k\}^n_{k=1} \subseteq \fA$ are unitary operators, $\{t_k\}^n_{k=1} \subseteq [0,1]$ are such that $\sum^n_{k=1} t_k = 1$, and $R = \sum^n_{k=1} t_k U^*_k T U_k$.  Then $R$ is self-adjoint and $\tau(R) = \tau(T)$.  Moreover, by parts (\ref{part:positive-multiple}, \ref{part:unitary}, \ref{part:concave-integral}) of Theorem \ref{thm:properties-of-eigenvalue-functions}, 
\begin{align*}
\int^t_0 \lambda^\tau_R(s) \, ds \leq \int^t_0 \sum^n_{k=1} t_k \lambda^\tau_{U^*_k TU_k}(s) \, ds = \int^t_0 \sum^n_{k=1} t_k \lambda^\tau_T(s) \ ds = \int^t_0 \lambda^\tau_T(s) \, ds
\end{align*}
for all $t \in [0,1]$.  Thus $R \prec_\tau T$ for all $R \in \conv(\U(T))$.

If $S \in \cconv(\U(T))$, then clearly $S = S^*$. The fact that $S \prec_\tau T$ then follows by part (\ref{part:difference}), the above paragraph, the fact that $\tau$ is norm continuous, and the fact that
\[
\left|\int^t_0 f(s) - g(s) \, ds\right| \leq \left\|f -g \right\|_\infty
\]
for all $t \in [0,1]$ and all bounded functions $f$ and $g$.
\end{proof}

It is unlikely that parts (\ref{part:positive-map}, \ref{part:cp-map}) of Theorem \ref{thm:majorization-in-factors} holds in arbitrary tracial C$^*$-algebras due to the lack of ability to take weak$^*$-limits of convex combinations of inner automorphisms.  However, we have the following analogue of \cite{H1987}*{Proposition 4.4}.
\begin{prop}
\label{prop:positive-maps}
Let $\fA$ be a unital C$^*$-algebra with a faithful tracial state $\tau$ and let $\varphi : \fA \to \fA$ be a positive map.  Then $\varphi$ is unital and $\tau$-preserving if and only if $\varphi(T) \prec_\tau T$ for all positive operators $T \in \fA$.
\end{prop}
\begin{proof}
Suppose $\varphi$ is unital, positive, and $\tau$-preserving.  Let $T \in \fA$ be positive.  Then $\tau(\varphi(T)) = \tau(T)$.  Furthermore, for all $r > 0$ notice
\[
\varphi(T) - r I_\fA = \varphi(T - rI_\fA) \leq \varphi((T - rI_\fA)_+)
\]
so
\[
\tau((\varphi(T) - r I_\fA)_+) \leq \tau(\varphi((T - rI_\fA)_+)) = \tau((T - r I_\fA)_+).
\]
Hence Corollary \ref{cor:versions-of-majorization-for-positives} implies that $\varphi(T) \prec_\tau T$.

Conversely, suppose $\varphi : \fA \to \fA$ is a positive map such that $\varphi(T) \prec_\tau T$ for all positive operators $T \in \fA$.  By part (\ref{part:trace}) of Theorem \ref{thm:properties-of-eigenvalue-functions}, 
\[
\tau(\varphi(T)) = \int^1_0 \lambda^\tau_{\varphi(T)}(s) \, ds = \int^1_0 \lambda^\tau_{T}(s) \, ds = \tau(T)
\]
for all positive operators $T \in \fA$.  Hence $\varphi$ is $\tau$-preserving.   Since $\lambda^\tau_{I_\fA}(s) = 1$ for all $s \in [0,1)$, by parts (\ref{part:decreasing-and-continuous}, \ref{part:norm-at-zero}) of Theorem \ref{thm:properties-of-eigenvalue-functions},
\[
\left\|\varphi(I_\fA)\right\| = \lim_{t \searrow 0} \frac{1}{t} \int^t_0 \lambda^\tau_{\varphi(I_\fA)}(s) \, ds \leq \lim_{t \searrow 0} \frac{1}{t} \int^t_0 \lambda^\tau_{I_\fA}(s) \, ds = 1.
\]
Hence $0 \leq \varphi(I_\fA) \leq I_\fA$.  If $\varphi(I_\fA) \neq I_\fA$, then
\[
0 = \tau(I_\fA) - \tau(\varphi(I_\fA)) = \tau(I_\fA - \varphi(I_\fA)) > 0,
\]
a clear contradiction.  Hence $\varphi(I_\fA) = I_\fA$.
\end{proof}

There are many other forms of majorization for elements of $L_\infty[0,1]$.  We note the following notion, which is used in Section \ref{sec:Other-Majorizations}.
\begin{defn}
\label{defn:absolute-submajorization}
Let $\fA$ be a unital C$^*$-algebra with a faithful tracial state $\tau$.  For $T, S \in \fA$, it is said that $T$ \emph{(absolutely) submajorizes} $S$ \emph{with respect to} $\tau$, denoted $S \prec^w_\tau T$, if
\begin{align*}
\int^t_0 \mu^\tau_S(s) \, ds &\leq \int^t_0 \mu^\tau_T(s) \, ds \text{ for all } t \in [0,1].
\end{align*}
\end{defn}

\section{Scalars in Convex Hulls}
\label{sec:Simplicity}

In this section, we will demonstrate for certain unital C$^*$-algebras $\fA$ with a faithful tracial state $\tau$ that $\tau(T)I_\fA \in \cconv(\U(T))$ for all self-adjoint $T \in \fA$ (see Theorem \ref{thm:scalars-in-convex-hulls}).  Combined with the Dixmier property, this implies these C$^*$-algebras are simple; that is, have no closed ideals (see Theorem \ref{thm:simple-C-algebras}).  We begin with definitions and examples of C$^*$-algebras for which these results apply.

\begin{defn}
A unital C$^*$-algebra $\fA$ is said to have real rank zero if the set of invertible self-adjoint operators of $\fA$ is dense in the set of self-adjoint operators.  Equivalently, by \cite{BP1991}, $\fA$ has real rank zero if and only if every self-adjoint element of $\fA$ can be approximated by self-adjoint elements with finite spectrum.  
Also $\fA$ is said to have stable rank one if the set of invertible elements is dense in $\fA$.
\end{defn}

\begin{defn}
Let $\fA$ be a unital C$^*$-algebra and let $P, Q \in \fA$ be projections.  It is said that $P$ and $Q$ are Murray-von Neumann equivalent (or simply equivalent), denoted $P \sim Q$, if there exists an element $V \in \fA$ such that $P = V^*V$ and $Q = VV^*$.  It is said that $P$ is equivalent to a subprojection of $Q$, denoted $P \lesssim Q$, if there exists a projection $Q' \leq Q$ such that $P \sim Q'$.
\end{defn}

\begin{defn}
Let $\fA$ be a unital C$^*$-algebra with a faithful tracial state $\tau$.  Then:
\begin{enumerate}
\item $\fA$ is said to have \emph{strong comparison of projections with respect to $\tau$} if for all projections $P, Q \in \fA$, $\tau(P) \leq \tau(Q)$ implies $P \lesssim Q$.
\item $\fA$ is said to have \emph{strict comparison of projections with respect to $\tau$} if for all projections $P, Q \in \fA$, $\tau(P) < \tau(Q)$ implies $P \lesssim Q$.
\end{enumerate}
\end{defn}

\begin{rem}
Note $\fA$ having strong (strict) comparison of projections with respect to $\tau$ is precisely saying that (FCQ1) (respectively (FCQ2)) of \cite{B1988} has an affirmative answer for $\fA$, provided $\tau$ is the only tracial state on $\fA$.  Furthermore, notice if $\fA$ has strong comparison of projections with respect to $\tau$, then $P \sim Q$ if and only if $\tau(P) = \tau(Q)$.  
\end{rem}

There are several C$^*$-algebras that are known to have the above properties.

\begin{exam}
\label{exam:properties0}
Type II$_1$ factors are well-known to be unital C$^*$-algebras that are simple, have real rank zero, and have strong comparison of projections with respect to a faithful tracial state, which happens to be unique.
\end{exam}
\begin{exam}
\label{exam:properties1}
It is not difficult to verify that UHF C$^*$-algebras and the Bunce-Deddens algebras (specific direct limits of $\M_n(C(\mathbb{T}))$) are unital, simple, real rank zero C$^*$-algebras that have strong comparison of projections with respect to a faithful tracial state, which happens to be unique.  However, as mentioned in \cite{B1988}, there exists unital, simple, AFD C$^*$-algebras with unique tracial states that do not have strong comparison of projections.
\end{exam}
\begin{exam}
\label{exam:properties2}
As mentioned in \cite{ST1992}, irrational rotation algebras and, more generally, simple non-commutative tori for which the map from K$_0$ to $\bR$ induced by the tracial state is faithful are examples of unital, simple, real rank zero C$^*$-algebras that have strong comparison of projections with respect to a faithful tracial state, which happens to be unique.
\end{exam}
\begin{exam}
\label{exam:properties3}
More generally, if $\fA$ is a unital, simple, C$^*$-algebra with real rank zero, stable rank one, and a tracial state $\tau$ such that the induced map $\tau_* : K_0(\fA) \to \bR$ defined by $\tau_*([x]_0) = \tau(x)$ is injective, then $\fA$ will have strong comparison of projections with respect to $\tau$ by cancellation.  In particular, \cite{E1993} can be used to produce examples.
\end{exam}
\begin{exam}
\label{exam:properties4}
In \cite{P2005}, it was demonstrated free minimal actions of $\mathbb{Z}^d$ on Cantor sets give rise to cross product C$^*$-algebras that have real rank zero, stable rank one, and strict comparison of projections with respect to their tracial states.
\end{exam}
\begin{exam}
\label{exam:properties5}
For certain tracial reduced free product C$^*$-algebras, \cite{A1982} implies simplicity, \cite{DHR1997} implies stable rank one, and \cite{DR2000} implies real rank zero and strict comparison of projections. 
\end{exam}

Notice that all of the C$^*$-algebras presented above are simple.  This turns out to be no coincidence.  To see this, we prove the following result.
\begin{thm}
\label{thm:scalars-in-convex-hulls}
Let $\fA$ be a unital C$^*$-algebra with real rank zero.  Suppose $\tau$ is a faithful tracial state on $\fA$ such that either:
\begin{enumerate}[(a)]
\item $\fA$ has strong comparison of projections with respect to $\tau$, or \label{ass:strong}
\item $\fA$ has strict comparison of projections with respect to $\tau$ and for every $n \in \bN$ there exists a projection $P \in \fA$ such that $0 < \tau(P) < \frac{1}{n}$. \label{ass:strict}
\end{enumerate}
Then $\tau(T) I_\fA \in \cconv(\U(T))$ for all self-adjoint $T \in \fA$.
\end{thm}

Once Theorem \ref{thm:scalars-in-convex-hulls} is established, we easily obtain the following.
\begin{thm}
\label{thm:simple-C-algebras}
If $\fA$ and $\tau$ are as in the hypotheses of Theorem \ref{thm:scalars-in-convex-hulls}, then $\fA$ is simple and $\tau$ is the unique tracial state on $\fA$.
\end{thm}
\begin{proof}
The following argument can be found in \cite{R1982} but is repeated for convenience of the reader.  Suppose $\I$ is a non-zero ideal in $\fA$.  Let $T \in \I \setminus \{0\}$ be positive.  Therefore $\tau(T) I_\fA \in \cconv(\U(T)) \subseteq \I$ by Theorem \ref{thm:scalars-in-convex-hulls}.  As $\tau$ is faithful, $\tau(T) \neq 0$ so $\I = \fA$.  Hence $\fA$ is simple.

Suppose $\tau_0$ is another tracial state on $\fA$.  By Lemma \ref{lem:easy-direction-for-convex-hulls}, $\tau_0(S) = \tau_0(T)$ for all $S \in \cconv(\U(T))$.  Hence Theorem \ref{thm:scalars-in-convex-hulls} implies 
\[
\tau_0(T) = \tau_0(\tau(T)I_\fA) = \tau(T).
\]
As this holds for all self-adjoint $T \in \fA$, we obtain that $\tau_0 = \tau$.
\end{proof}
\begin{rem}
If $\fA$ is a unital, infinite dimensional C$^*$-algebra with real rank zero and a faithful tracial state $\tau$, then it is possible to verify for all $n \in \bN$ that there exists a projection $P \in \fA$ such that $0 < \tau(P) < \frac{1}{n}$.
\end{rem}

\begin{exam}
To see why strict comparison of projections without arbitrarily small projections is not sufficient in Theorem \ref{thm:simple-C-algebras}, consider the C$^*$-algebra $\fA = \bC \oplus \bC$ with the faithful tracial state $\tau((a,b)) = \frac{1}{2}(a+b)$.  It is clear that $\fA$ is a unital C$^*$-algebra with real rank zero and strict comparison of projections with respect to $\tau$. However, $\fA$ is not simple.
\end{exam}

\begin{rem}
There are non-simple C$^*$-algebras with faithful tracial states.  Indeed \cite{M2000} produces a unital non-separable C$^*$-algebra with a faithful tracial state whereas \cite{O2013} produces a unital, separable, nuclear, non-simple C$^*$-algebra with a faithful tracial state.
\end{rem}

Note the following easily verified lemma which will be used often without citation.
\begin{lem}
\label{lem:convex-hull-transitive}
Let $\fA$ be a unital C$^*$-algebra and let $T, S, R \in \fA$.  If $T \in \cconv(\U(S))$ and $S \in \cconv(\U(R))$, then  $T \in \cconv(\U(R))$.
\end{lem}

To prove Theorem \ref{thm:scalars-in-convex-hulls}, it will suffice to prove the theorem for self-adjoint operators with finite spectrum by the assumption that $\fA$ has real rank zero.  Combined with the following remark, it will suffice to consider self-adjoint operators with two points in their spectra.
\begin{rem}
Let $\fA$ be a unital C$^*$-algebra and let $P \in \fA$ be a non-zero projection.  If $\fA$ has real rank zero, then $P\fA P$ is a unital C$^*$-algebra of real rank zero by \cite{BP1991}.  Furthermore, if $\tau$ is a faithful tracial state on $\fA$ satisfying hypothesis (\ref{ass:strong}) (respectively (\ref{ass:strict})) of Theorem \ref{thm:scalars-in-convex-hulls}, then $\tau_P : P \fA P \to \bC$ defined by $\tau_P(PTP) = \frac{1}{\tau(P)} \tau(PTP)$ is a faithful tracial state on $P\fA P$ satisfying hypothesis (\ref{ass:strong}) (respectively (\ref{ass:strict})) of Theorem \ref{thm:scalars-in-convex-hulls}.  Thus the hypotheses of Theorem \ref{thm:scalars-in-convex-hulls} are all preserved under compressions.  We will continue throughout the remainder of the paper to use $\tau_P$ for the tracial state defined above.
\end{rem}

To prove Theorem \ref{thm:scalars-in-convex-hulls} for self-adjoint operators with two points in their spectra, we will use equivalence of projections to construct matrix algebras and apply results on majorization for self-adjoint matrices, specifically part (\ref{part:convex}) of Theorem \ref{thm:majorization-in-factors}, to average part of one spectral projection with the other.  Using a back-and-forth-type argument, we eventually obtain an operator in $\conv(\U(T))$ that is almost $\tau(T)I_\fA$. 

As $\tau(\fA)$ may not equal $[0,1]$, we may only divide projections up based on the size of another projection.  As such, the following division algorithm result will be of use to us and is easily verified.
\begin{lem}
\label{lem:division-algorithm}
Let $t \in (0, \frac{1}{2}]$ and write $1 = k_1 t + r_1$ where $k_1 \in \bN$ and $0 \leq r_1 < t$.  Then $k_1 \geq 2$ and $0 \leq r_1 < \frac{1}{k_1+1}$.  Furthermore, if $r_1 \neq 0$ and $1 = k_2 r_1 + r_2$ for some $k_2 \in \bN$ and $0 \leq r_2 \leq r_1$, then $k_2 \geq k_1$.
\end{lem}

The following lemma will be our method of constructing matrix algebras.  However, the embedding of each matrix algebra into $\fA$ need not be a unital embedding.
\begin{lem}
\label{lem:constructing-matrix-algebras}
Let $\fA$ be a unital C$^*$-algebra with a faithful tracial state $\tau$ and let $P \in \fA$ be a projection with $\tau(P) \in \left(0,\frac{1}{2}\right]$.    Write $1 = k \tau(P) + r$ where $k \in \bN$ and $0 \leq r < \tau(P)$.  If
\begin{enumerate}
\item $\fA$ has strong comparison of projections with respect to $\tau$ and $\ell = k-1$, or
\item $\fA$ has strict comparison of projections with respect to $\tau$, $r \neq 0$, and $\ell = k-1$, or \label{case:strict}
\item $\fA$ has strict comparison of projections with respect to $\tau$, and $\ell = k-2$, \label{case:strict-one-less}
\end{enumerate}
then there exists pairwise orthogonal subprojections $\{P_j\}^{\ell}_{j=1}$ of $I_\fA - P$ such that $\{P\} \cup \{P_j\}^{\ell}_{j=1}$ are equivalent in $\fA$.
\end{lem}
\begin{proof}
Notice $\tau(I_\fA - P) = (k-1) \tau(P) + r$.  Since $k\geq 2$, $\tau(P) \leq \tau(I_\fA - P)$ with strict inequality when $r \neq 0$.  Therefore, by assumptions, there exists a subprojection $P_1$ of $I_\fA - P$ such that $P_1 \sim P$.  If $k \geq 3$ (and $\ell \geq 2$), there exists a subprojection $P_2$ of $I_\fA - P - P_1$ such that $P_2 \sim P$.  By repeating this argument, we obtain pairwise orthogonal subprojections $\{P_j\}^{\ell}_{j=1}$ of $I_\fA - P$ such that $P_j \sim P$ for all $j$.  As Murray-von Neumann equivalence is an equivalence relation, the result follows.
\end{proof}

We now divide the prove of Theorem \ref{thm:scalars-in-convex-hulls} for $T$ with two point spectra into two parts: Lemma \ref{lem:two-point-spectrum-strong} proves the result when $\fA$ has strong comparison of projections, and Lemma \ref{lem:two-point-spectrum-strict} will modify the argument to obtain the result in the other case.  In that which follows, $\diag(a_1, \ldots, a_n)$ denotes the diagonal $n \times n$ matrix with diagonal entries $a_1, \ldots, a_n$.

\begin{lem}
\label{lem:two-point-spectrum-strong}
Let $\fA$ be a unital C$^*$-algebra with real rank zero that has strong comparison of projections with respect to a faithful tracial state $\tau$.  If $P \in \fA$ is a projection, $a,b \in \bR$, and $T = a P + b (I_\fA - P)$, then $\tau(T) I_\fA \in \cconv(\U(T))$.
\end{lem}
\begin{proof}
By interchanging $P$ and $I_\fA - P$, we may assume that $\tau(P) \leq \frac{1}{2}$.  Let $r_0 = \tau(P)$ and write $1 = k_1 r_0 + r_1$ where $k_1 \in \bN$, $k_1 \geq 2$, and $0 \leq r_1 \leq \min\{r_0, \frac{1}{k_1+1}\} < \frac{1}{2}$.  By Lemma \ref{lem:constructing-matrix-algebras} there pairwise orthogonal subprojections $\{Q_j\}^{k_1-1}_{j=1}$ of $I_\fA - P$ such that $\{P\} \cup \{Q_j\}^{k-1}_{j=1}$ are equivalent in $\fA$.  Let $P_1 = I_\fA - P - \sum^{k_1-1}_{j=1} Q_j$.  Using the equivalence of $\{P\} \cup \{Q_j\}^{k-1}_{j=1}$, a copy of $\M_{k_1}(\bC)$ may be constructed  in $\fA$ with unit $I_\fA - P_1$.  Using this matrix subalgebra, $T$ can be viewed as the operator
\[
T = \diag(a, b, \ldots, b) \oplus b P_1 \in \M_{k_1}(\bC) \oplus P_1 \fA P_1 \subseteq \fA.
\]
Since any self-adjoint matrix majorizes its normalized trace (see Lemma \ref{lem:easy-direction-for-convex-hulls}), we obtain by Theorem \ref{thm:majorization-in-factors} that
\[
\frac{a + (k_1-1)b}{k_1} I_{k_1}\in \cconv(\U(\diag(a, b, \ldots, b)))
\]
where the unitary orbit is computed in $\M_{k_1}(\bC)$.  Therefore, if $a_1 = \frac{a + (k_1-1)b}{k_1}$, we obtain by using a direct sum argument that
\[
T_1 := a_1(I_\fA - P_1) + bP_1 \in \cconv(\U(T)).
\]

Notice $\tau(P_1) = r_1$.  If $r_1 = 0$, the proof is complete (as $\tau(T_1) = \tau(T)$).  Otherwise, by writing $1 = k_2 r_1 + r_2$ where $k_2 \in \bN$, $k_2 \geq k_1$, and $0 \leq r_2 \leq \min\{r_1, \frac{1}{k_2+1}\}$, and by repeating the above argument, there exists a projection $P_2 \in \fA$ such that $\tau(P_2) = r_2$ and
\[
T_2 :=  a_1 P_2 + \frac{b + (k_2-1)a_1}{k_2}(I_\fA - P_2) \in  \cconv(\U(T_1)) \subseteq \cconv(\U(T)).
\]
Notice if $r_2 = 0$, the proof is again complete.

Repeat the above process ad infinitum.  Notice that the proof is complete if the process ever terminates via a zero remainder.  As such, we may assume that we have found a non-decreasing sequence $(k_n)_{n\geq 1} \subseteq \bN$ with $k_1 \geq 2$, a sequence $(r_n)_{n\geq 1} \subseteq \left(0, \frac{1}{2}\right]$ with $1 = k_{n+1} r_n + r_{n+1}$, projections $\{P_n\}_{n\geq 1} \subseteq \fA$ with $\tau(P_n) = r_n$, sequences $(a_n)_{n\geq 1}, (b_n)_{n\geq 1} \subseteq \bR$ such that
\[
a_{n+1} = \frac{a_n + (k_{2n+1} - 1)b_n}{k_{2n+1}} \qqand b_{n+1} = \frac{b_n + (k_{2n+2} - 1)a_{n+1}}{k_{2n+2}},
\] 
and operators
\[
T_{2n} = a_n P_{2n} + b_n (I_\fA - P_{2n}) \qqand T_{2n+1} = b_n P_{2n+1} + a_{n+1} (I_\fA - P_{2n+1})
\]
such that $T_n \in \cconv(\U(T))$ for all $n$.  

If $a \leq b$, it is elementary to verify that
\[
a \leq a_1 \leq a_2 \leq \cdots \leq b_2 \leq b_1 \leq b.
\]
Similarly, if $b \leq a$, then
\[
b \leq b_1 \leq b_2 \leq \cdots \leq a_2 \leq a_1 \leq a.
\]
As a result, $(a_n)_{n\geq 1}$ and $(b_n)_{n\geq 1}$ are bounded monotone sequence of $\bR$ and thus converge.  Let
\[
a' = \lim_{n\to \infty} a_n \qqand b' = \lim_{n\to \infty} b_n.
\]
If the non-decreasing sequence $(k_n)_{n\geq 1}$ is bounded, using the fact that $k_1 \geq 2$ and the relations between $a_n$ and $b_n$, we obtain $a' = b'$.  If $(k_n)_{n\geq 1}$ is unbounded, then by using the fact that
\[
\lim_{m\to \infty} \left| c - \frac{c + md}{m+1}\right| = |c-d|
\]
we again obtain $a' = b'$.

Let $\epsilon > 0$ and choose $n$ such that $|a_n - a'| < \epsilon$ and $|b_n - a'| < \epsilon$.  Then $\left\|T_{2n} - a' I_\fA\right\| < \epsilon$ so
\[
\dist\left(a' I_\fA, \cconv(\U(T))\right) \leq \epsilon.
\]
Hence $a' I_\fA \in \cconv(\U(T))$.  Since every element of $\cconv(\U(T))$ has trace equal to $\tau(T)$, we obtain $a' = \tau(T)$ thereby completing the result.
\end{proof}

\begin{lem}
\label{lem:two-point-spectrum-strict}
Let $\fA$ be a unital C$^*$-algebra with real rank zero and property (\ref{ass:strict}) of Theorem \ref{thm:scalars-in-convex-hulls} with respect to a faithful tracial state $\tau$.  If $P \in \fA$ is a projection, $a,b \in \bR$, and $T = a P + b (I_\fA - P)$, then $\tau(T) I_\fA \in \cconv(\U(T))$.
\end{lem}
\begin{proof}
Notice, by case (\ref{case:strict}) of Lemma \ref{lem:constructing-matrix-algebras}, that the recursive algorithm in the proof of Lemma \ref{lem:two-point-spectrum-strong} works at the $n^{\mathrm{th}}$ stage in this setting provided $r_n \neq 0$. Therefore, if $r_n \neq 0$ for all $n \in \bN$, the proof is complete.  Otherwise, if $n$ is the first number in the algorithm for which $r_n = 0$, notice $r_{n-1} = \frac{1}{k_n}$.  Thus it suffices to prove the result in the case that $\tau(P) = \frac{1}{k}$ for some $k \in \bN$ with $k \geq 2$.

If $k \geq 3$, we can apply the algorithm in Lemma \ref{lem:two-point-spectrum-strong} by viewing the remainders as being $\frac{1}{k}$ instead of zero.  Indeed the proof of Lemma \ref{lem:two-point-spectrum-strong} may be adapted using case (\ref{case:strict-one-less}) instead of case (\ref{case:strict}) of Lemma \ref{lem:constructing-matrix-algebras} to construct $(k_n-1) \times (k_n - 1)$ matrix algebras (instead of $k_n \times k_n$) and by using the new scalars
\[
a_{n+1} = \frac{a_n + (k_{2n+1} - 2)b_n}{k_{2n+1}-1} \qqand b_{n+1} = \frac{b_n + (k_{2n+2} - 2)a_{n+1}}{k_{2n+2}-1}.
\] 
The remainder of the proof then follows as in Lemma \ref{lem:two-point-spectrum-strong}.  Thus it remains to prove the result in the case $\tau(P) = \frac{1}{2}$.

Since $\fA$ has property (\ref{ass:strict}), there exists a projection $P_0 \leq I_\fA - P$ with $\tau(P_0) < \frac{1}{2}$.  Consider
\[
T_0 = aP + b P_0 \in (P + P_0) \fA (P + P_0).
\]
As $(P + P_0) \fA (P + P_0)$ satisfies the assumptions of this lemma and since 
\[
\tau_{(P + P_0)}(P) = \frac{1}{\tau(P+P_0)} \tau(P) \neq \frac{1}{2},
\]
the above cases imply there exists $\alpha_0 \in \bR$ such that $\alpha_0 (P + P_0) \in \cconv(\U(T_0))$ where $\cconv(\U(T_0))$ is computed in $(P + P_0) \fA (P + P_0)$.  Consequently
\[
\alpha_0 (P + P_0) + b (I_\fA - P - P_0) \in \cconv(\U(T))
\]
by a direct sum argument.  As $\tau(P + P_0) \neq \frac{1}{2}$, the above cases imply there exists $\alpha \in \bR$ such that $\alpha I_\fA\in \cconv(\U(T))$.  As every element of $\cconv(\U(T))$ has trace $\tau(T)$, $\alpha = \tau(T)$ completing the result.
\end{proof}

\begin{lem}
\label{lem:trace-for-finite-spectrum}
Let $\fA$ and $\tau$ be as in the hypotheses of Theorem \ref{thm:scalars-in-convex-hulls}.  If $T \in \fA$ is a self-adjoint operator with finite spectrum, then $\tau(T) I_\fA \in \cconv(\U(T))$.
\end{lem}
\begin{proof}
By assumption there exist pairwise orthogonal non-zero projections $\{P_k\}^n_{k=1}$ and scalars $\{\alpha_k\}^n_{k=1} \subseteq \bR$ such that $T = \sum^n_{k=1} \alpha_k P_k$.  By applying Lemma \ref{lem:two-point-spectrum-strong} or \ref{lem:two-point-spectrum-strict} to $\alpha_1 P_1 + \alpha_2 P_2$ in $(P_1 +P_2) \fA (P_1 + P_2)$ and by appealing to a direct sum argument, there exists a $\beta_0 \in \bR$ such that
\[
\beta_0 (P_1 + P_2) + \sum^n_{k=3} \alpha_k P_k \in \cconv(\U(T)).
\]
By iterating this argument another $n - 2$ times, there exists a $\beta \in \bR$ such that $\beta I_\fA \in \cconv(\U(T))$.  As every element of $\cconv(\U(T))$ has trace $\tau(T)$, $\beta = \tau(T)$ completing the result.
\end{proof}

\begin{proof}[Proof of Theorem \ref{thm:scalars-in-convex-hulls}]
Let $T \in \fA$ be self-adjoint.  Let $\epsilon > 0$.  Since $\fA$ has real rank zero, there exists a self-adjoint operator $T_0 \in \fA$ with finite spectrum such that $\left\|T - T_0\right\| < \epsilon$.  Notice this implies $\dist\left(R, \cconv(\U(T))\right) \leq \epsilon$ for all $R \in \cconv(\U(T_0))$.

By Lemma \ref{lem:trace-for-finite-spectrum}, $\tau(T_0) I_\fA \in \cconv(\U(T_0))$.  Since $|\tau(T_0) - \tau(T)| < \epsilon$, we obtain
\[
\dist\left(\tau(T) I_\fA, \cconv(\U(T))\right) < 2\epsilon.
\]
As $\epsilon$ was arbitrary, the result follows.
\end{proof}

\begin{rem}
\label{rem:short-proof-of-simple}
Using the above ideas, there is a simple proof that an infinite dimensional C$^*$-algebra satisfying the assumptions of Theorem \ref{thm:scalars-in-convex-hulls} must be simple.  Indeed suppose $\fA$ is such a C$^*$-algebra and $\I$ is a non-zero ideal.  Note $\I$ is hereditary and thus has real rank zero as hereditary C$^*$-subalgebras of $\fA$ have real rank zero (see \cite{BP1991}*{Corollary 2.8}).  Thus the unitization of $\I$ contains a non-zero projection and thus $\I$ contains a non-zero projection. 

Note the set of projections contained in $\I$ is closed under taking subprojections (as $\I$ is hereditary) and is closed under Murray-von Neumann equivalence (as $\I$ is an ideal).  Therefore, by part (\ref{case:strict-one-less}) of Lemma \ref{lem:constructing-matrix-algebras}, there exists a projection $P \in \I$ with $\tau(P) \geq \frac{1}{2}$.

If $\tau(P) = \frac{1}{2}$, choose a non-zero projection $P' \leq P$ with $\tau(P') < \frac{1}{2}$ and a subprojection $Q$ of $I_\fA - P$ with $\tau(Q) = \tau(P')$ such that $Q \sim P'$.  Hence $Q \in \I$ so $P + Q \in \I$.  As $\tau(P + Q) > \frac{1}{2}$, we have reduced to the case $\tau(P) > \frac{1}{2}$.

If $\tau(P) > \frac{1}{2}$, then $I_\fA - P$ is equivalent to a subprojection of $P$ and thus $I_\fA - P \in \I$.  Since $P \in \I$, this implies $I_\fA \in \I$ so $\I = \fA$. 
\end{rem}

\section{Convex Hulls of Unitary Orbits}
\label{sec:Convex-Hulls-Of-Unitary-Orbits}

In this section, we will demonstrate the following theorem which characterizes $\cconv(\U(T))$ for self-adjoint $T$ in various C$^*$-algebras using the notion of majorization.

\begin{thm}
\label{thm:classification-of-convex-hull}
Let $\fA$ be a unital C$^*$-algebra with real rank zero that has strong comparison of projections with respect to a faithful tracial state $\tau$.   If $T \in \fA$ is self-adjoint, then
\[
\cconv(\U(T)) = \{S \in \fA \, \mid \, S^* = S, S \prec_\tau T\}.
\]
\end{thm}

Before proceeding, we briefly outline the approach to the proof.  First, we reduce to the case that $T$ and $S$ have finite spectrum.  This is done by showing $T$ and $S$ can be approximated by self-adjoint operators $T'$ and $S'$ such that $S' \prec_\tau T'$.  We then demonstrate a `pinching' on self-adjoint operators $T'$ with exactly two points in their spectrum to show that all convex combinations of $T'$ and $\tau(T')I_\fA$ are in $\cconv(\U(T'))$.  Appealing to a specific decomposition result and by progressively applying pinchings, the result is obtained. 

We begin with the decomposition result.

\begin{lem}
\label{lem:getting-equal-trace-projections}
Let $\fA$ and $\tau$ be as in Theorem \ref{thm:classification-of-convex-hull}.  Suppose $S, T \in \fA$ are self-adjoint operators with finite spectrum.  Then there exists two collections of pairwise orthogonal non-zero projections $\{P_k\}^n_{k=1}$ and $\{Q_k\}^n_{k=1}$ with 
\[
\sum^n_{k=1} P_k= \sum^n_{k=1} Q_k = I_\fA \qand
\tau(P_k) = \tau(Q_k) \text{ for all }k
\]
and scalars $\{\alpha_k\}^n_{k=1}, \{\beta_k\}^n_{k=1} \subseteq \bR$ with $\alpha_k \geq \alpha_{k+1}$ and $\beta_k \geq \beta_{k+1}$ such that
\[
T = \sum^n_{k=1} \alpha_k P_k \qqand S = \sum^n_{k=1} \beta_k Q_k.
\]
\end{lem}
\begin{proof}
Since $T$ and $S$ have finite spectrum, there exists two collections of pairwise orthogonal non-zero projections $\{P'_k\}^m_{k=1}$ and $\{Q'_k\}^l_{k=1}$ with $\sum^m_{k=1} P_k= \sum^l_{k=1} Q_k = I_\fA$ and scalars $\{\alpha'_k\}^m_{k=1}, \{\beta_k\}^l_{k=1} \subseteq \bR$ with $\alpha'_k > \alpha'_{k+1}$ and $\beta'_k > \beta'_{k+1}$ such that
\[
T = \sum^m_{k=1} \alpha'_k P'_k \qqand S = \sum^l_{k=1} \beta'_k Q'_k.
\]

Suppose $\tau(P'_1) \geq \tau(Q'_1)$.  Since $\fA$ has strong comparison of projections, there exists a projection $P_1 \in \fA$ such that $\tau(P_1) = \tau(Q'_1)$ and $P_1\leq P'_1$.  Letting $Q_1 = Q'_1$, we have
\[
T = \alpha'_1 P_1 + \alpha'_1 (P'_1 - P_1) + \sum^m_{k=2} \alpha'_k P'_k \qqand S = \beta'_1Q_1 + \sum^l_{k=2} \beta'_k Q'_k.
\]
Similarly, if $\tau(P'_1) \leq \tau(Q'_1)$, there exists a projection $Q_1 \in \fA$ such that $\tau(Q_1) = \tau(P'_1)$ and $Q_1\leq Q'_1$.  Letting $P_1 = P'_1$, we have
\[
T = \alpha'_1 P_1 + \sum^m_{k=2} \alpha'_k P'_k \qqand S = \beta'_1Q_1  + \beta'_1 (Q'_1 - Q_1)+ \sum^l_{k=2} \beta'_k Q'_k.
\]
By repeating this argument at most another $m + l - 1$ times (for the next iteration, using $P'_2$ and $Q'_2$ when $\tau(P'_1) = \tau(Q'_1)$ and otherwise using $P'_1 - P_1$ and $Q'_2$ in the first case and $P'_2$ and $Q'_1 - Q_1$ in the second case), the result follows.
\end{proof}

The following result enables us to reduce Theorem \ref{thm:classification-of-convex-hull} to the case of self-adjoint operators with finite spectrum.  More is demonstrated than is needed for Theorem \ref{thm:classification-of-convex-hull} in order to facilitate results in Section \ref{sec:Other-Majorizations}.

\begin{lem}
\label{lem:reduction-to-finite-spectrum}
Let $\fA$ and $\tau$  be as in Theorem \ref{thm:classification-of-convex-hull}.   If $S, T \in \fA$ are self-adjoint operators, then for every $\epsilon > 0$ there exists self-adjoint operators $S', T' \in \fA$ with finite spectrum such that
\[
\left\|T - T'\right\| < \epsilon, \qqand \left\|S - S'\right\| < \epsilon.
\]
Furthermore:
\begin{enumerate}
\item $T' \prec_\tau T$ and $S' \prec_\tau S$.
\item If $S, T \geq 0$, then $S', T' \geq 0$.
\item If $S \prec_\tau T$, then $S' \prec_\tau T'$.
\item If $S, T \geq 0$ and $S \prec^w_\tau T$, then  $S' \prec^w_\tau T'$.
\item If $\lambda^\tau_{S}(s) \leq \lambda^\tau_{T}(s)$ for all $s \in [0,1)$, then $\lambda^\tau_{S'}(s) \leq \lambda^\tau_{T'}(s)$ for all $s \in [0,1)$.
\end{enumerate}
\end{lem}
\begin{proof}
Let $\epsilon > 0$.  Since $\fA$ has real rank zero, there exists self-adjoint operators $T_0, S_0 \in \fA$ with finite spectrum such that 
\[
\left\|T - T_0\right\| \leq \frac{\epsilon}{2} \qqand \left\|S - S_0\right\| \leq \frac{\epsilon}{2}.
\]
Let $\{P_k\}^n_{k=1}$, $\{Q_k\}^n_{k=1}$, $\{\alpha_k\}^n_{k=1}$, and $\{\beta_k\}^n_{k=1}$ be as in the conclusions of Lemma \ref{lem:getting-equal-trace-projections} so that
\[
T_0 = \sum^n_{k=1} \alpha_k P_k \qqand S_0 = \sum^n_{k=1} \beta_k Q_k,
\]
and, for each $k \in \{0, 1,\ldots, n\}$, let $s_k = \sum^k_{j=1} \tau(P_j)$.  Notice $s_k < s_{k+1}$ for all $k$, $s_0 = 0$, $s_n = 1$, and $\lambda^\tau_{T_0}(s) = \alpha_k$ and $\lambda^\tau_{S_0}(s) = \beta_k$ for all $s \in [s_{k-1}, s_k)$ by Example \ref{exam:finite-spectrum-eigenvalue-functions}.   For each $k \in \{1,\ldots, n\}$, let
\[
\alpha'_k = \frac{1}{s_k - s_{k-1}} \int^{s_k}_{s_{k-1}} \lambda^\tau_{T}(s) \, ds \qqand \beta'_k = \frac{1}{s_k - s_{k-1}} \int^{s_k}_{s_{k-1}} \lambda^\tau_{S}(s) \, ds,
\]
and let 
\[
T' = \sum^n_{k=1} \alpha'_k P_k \qqand S' = \sum^n_{k=1} \beta'_k Q_k.
\]

We claim $T'$ and $S'$ are the desired self-adjoint operators.  Indeed Example \ref{exam:averaging} implies $T' \prec_\tau T$ and $S' \prec_\tau S$.  Furthermore, if $S, T \geq 0$, then $\lambda^\tau_{S}(s)$ and $\lambda^\tau_T(s)$ are non-negative functions by part (\ref{part:norm-at-zero}) of Theorem \ref{thm:properties-of-eigenvalue-functions}.  Consequently $\alpha'_k, \beta'_k \geq 0$ for all $k$, so $S', T' \geq 0$.

To see that $\left\|T - T'\right\| < \epsilon$, it suffices to show that $\left\|T_0 - T'\right\| \leq \frac{\epsilon}{2}$.  For each $k$, notice
\begin{align*}
\left|\alpha_k - \alpha'_k\right| &\leq \frac{1}{s_k - s_{k-1}} \int^{s_k}_{s_{k-1}} |\alpha_k - \lambda^\tau_T(s)| \, ds \\
&= \frac{1}{s_k - s_{k-1}} \int^{s_k}_{s_{k-1}} |\lambda^\tau_{T_0}(s) - \lambda^\tau_T(s)| \, ds \\
&\leq \frac{1}{s_k - s_{k-1}} \int^{s_k}_{s_{k-1}} \left\|T_0 - T\right\| \, ds= \left\|T_0 - T\right\| < \frac{\epsilon}{2}
\end{align*}
by part (\ref{part:difference}) of Theorem \ref{thm:properties-of-eigenvalue-functions}.  As this holds for all $k$, we obtain $\left\|T_0 - T'\right\| \leq \frac{\epsilon}{2}$.  The same arguments show $\left\|S - S'\right\| < \epsilon$.

Suppose $S \prec_\tau T$.  Notice, by part (\ref{part:decreasing-and-continuous}) of Theorem \ref{thm:properties-of-eigenvalue-functions}, that $\alpha'_k \geq \alpha'_{k+1}$ and $\beta'_k \geq \beta'_{k+1}$ for all $k$.  Consequently $\lambda^\tau_{T'}(s) = \alpha'_k$ and $\lambda^\tau_{S'}(s) = \beta'_k$ for all $s \in [s_{k-1}, s_k)$ by Example \ref{exam:finite-spectrum-eigenvalue-functions}.   This along with the definition of $\alpha'_k$ and $\beta'_k$ implies
\[
\int^{s_k}_{s_{k-1}} \lambda^\tau_{T'}(s) \, ds = \int^{s_k}_{s_{k-1}} \lambda^\tau_T(s) \, ds \qqand\int^{s_k}_{s_{k-1}} \lambda^\tau_{S'}(s) \, ds = \int^{s_k}_{s_{k-1}} \lambda^\tau_S(s) \, ds
\]
for all $k$.   In particular, by adding integrals, we obtain
\[
\int^{1}_{0} \lambda^\tau_{T'}(s) \, ds = \int^{1}_{0} \lambda^\tau_T(s) \, ds =\int^{1}_{0} \lambda^\tau_S(s) \, ds =    \int^{1}_{0} \lambda^\tau_{S'}(s) \, ds.
\]
For an arbitrary $t\in [0,1]$, choose $k \in \{1,\ldots, n\}$ such that $t \in [s_{k-1}, s_{k}]$ and notice
\[
\int^t_0 \lambda^\tau_{T'}(s) - \lambda^\tau_{S'}(s) \, ds = \int^{s_{k-1}}_0 \lambda^\tau_{T}(s) - \lambda^\tau_{S}(s) \, ds + \int^{t}_{s_{k-1}} \lambda^\tau_{T'}(s) - \lambda^\tau_{S'}(s) \, ds.
\]
To see the left-hand-side is always non-negative, we note that $\lambda^\tau_{T'}(s) - \lambda^\tau_{S'}(s)$ is constant on $[s_{k-1}, s_k)$.  If $\lambda^\tau_{T'}(s) - \lambda^\tau_{S'}(s) \geq 0$ on $[s_{k-1}, s_k)$, then
\[
\int^t_0 \lambda^\tau_{T'}(s) - \lambda^\tau_{S'}(s) \, ds \geq  \int^{s_{k-1}}_0 \lambda^\tau_{T}(s) - \lambda^\tau_{S}(s) \, ds \geq 0.
\]
Otherwise $\lambda^\tau_{T'}(s) - \lambda^\tau_{S'}(s) < 0$ on $[s_{k-1}, s_k)$ so
\begin{align*}
\int^t_0 \lambda^\tau_{T'}(s) - \lambda^\tau_{S'}(s) \, ds &\geq \int^{s_{k-1}}_0 \lambda^\tau_{T}(s) - \lambda^\tau_{S}(s) \, ds + \int^{s_k}_{s_{k-1}} \lambda^\tau_{T'}(s) - \lambda^\tau_{S'}(s) \, ds\\
&= \int^{s_{k-1}}_0 \lambda^\tau_{T}(s) - \lambda^\tau_{S}(s) \, ds + \int^{s_k}_{s_{k-1}} \lambda^\tau_{T}(s) - \lambda^\tau_{S}(s) \, ds \geq 0
\end{align*}
Hence $S' \prec_\tau T'$ when $S \prec_\tau T$.

If $S, T \geq 0$ and $S \prec^w_\tau T$, then the proof that $S' \prec^w_\tau T'$ follows from the above proof (ignoring the part that shows $\int^1_0 \lambda^\tau_{S'}(s) \, ds = \int^1_0 \lambda^\tau_{T'}(s) \, ds$).

If $\lambda^\tau_{S}(s) \leq \lambda^\tau_{T}(s)$ for all $s \in [0,1)$, then $\beta'_k \leq \alpha'_k$ for all $k$ and thus $\lambda^\tau_{S'}(s) \leq \lambda^\tau_{T'}(s)$ for all $s \in [0,1)$ by Example \ref{exam:finite-spectrum-eigenvalue-functions}.
\end{proof}

The following result for elements of $\M_2(\bC)$ is referred to as a pinching.
\begin{lem}
\label{lem:pinching}
Let $\fA$ and $\tau$  be as in Theorem \ref{thm:classification-of-convex-hull}.   If $P \in \fA$ is a projection, $a,b \in \bR$, and $T = a P + b (I_\fA - P)$, then for all $t \in [0,1]$, 
\[
tT + (1-t) \tau(T) I_\fA = (at + \tau(T)(1-t)) P + (bt + \tau(T)(1-t)) (I_\fA - P) \in \cconv(\U(T)).
\]
\end{lem}
\begin{proof}
Fix $t \in[0,1]$ and let 
\[
a' = at + \tau(T)(1-t) \qqand b' = bt + \tau(T)(1-t).
\]
Since $\tau(T) = a \tau(P) + b \tau(I_\fA - P) \in \conv(\{a,b\})$, we obtain that $a', b' \in \conv(\{a,b\})$.  

By interchanging $P$ and $I_\fA - P$, we may assume that $\tau(P) \leq \frac{1}{2}$.  Since $\fA$ has strong comparison of projections, there exists a projection $Q \in \fA$ such that $Q \sim P$ and $Q \leq I_\fA - P$.  Consequently, using the partial isometry implementing the equivalence of $P$ and $Q$, a copy of $\M_2(\bC)$ may be constructed in $(P+Q)\fA (P+Q)$ so that $P$ and $Q$ are the two diagonal rank one projections.  Hence $T$ can be viewed as the operator
\[
T = (aP +bQ) \oplus b (I_\fA - P - Q) \in \M_2(\bC) \oplus (I_\fA - P - Q) \fA (I_\fA - P - Q) \subseteq \fA.
\]

Choose $b'' \in \bR$ so that $b'' + a' = a + b$.  Notice $b'' \in \conv(\{a, b\})$ as $a' \in \conv(\{a,b\})$.  Using Example \ref{exam:majorization-self-adjoint-matrices}, we see that
\[
\diag(a', b'') \prec_{\frac{1}{2} \tr} \diag(a, b)
\]
where $\frac{1}{2} \tr$ is the normalized trace on $\M_2(\bC)$ (which agrees with $\tau_{P+Q}$).  Thus Theorem \ref{thm:majorization-in-factors} along with a direct sum argument implies that
\[
a'P + b''Q + b (I_\fA - P - Q) \in \cconv(\U(T)).
\]

By applying Theorem \ref{thm:scalars-in-convex-hulls} to $b''Q + b (I_\fA - P - Q)$ in $(I_\fA - P)\fA (I_\fA - P)$ and by applying a direct sum argument, we obtain that
\[
a' P + b''' (I_\fA - P) \in \cconv(\U(T))
\]
for some $b''' \in \bR$.  As every element of $\cconv(\U(T))$ has trace $\tau(T)$, one can verify that $b''' = b'$.
\end{proof}

The following result contains the main technical details necessary for a recursive argument in the proof of Theorem \ref{thm:classification-of-convex-hull}.  In particular, it will enable us to systematically apply pinchings.

\begin{lem}
\label{lem:reduction-argument}
Let $\fA$ and $\tau$ be as in Theorem \ref{thm:classification-of-convex-hull}.  Suppose  $\{P_k\}^n_{k=1}$ is a collection of pairwise orthogonal projections with $\sum^n_{k=1} P_k= I_\fA$, $\{\alpha_k\}^n_{k=1}, \{\beta_k\}^n_{k=1} \subseteq \bR$ with $\beta_k \geq \beta_{k+1}$ for all $k$, and
\[
T = \sum^n_{k=1} \alpha_k P_k \qqand S = \sum^n_{k=1} \beta_k P_k.
\]
Suppose further that $S \prec_\tau T$ and there exists a $j$ such that $\alpha_k \geq \beta_1$ for all $k < j$, $\alpha_j < \beta_1$, and $\alpha_k \geq \alpha_{k+1}$ for all $k \geq j$.  Then there exists $\{\alpha'_k\}^n_{k=1} \subseteq \bR$ such that $\alpha'_1 = \beta_1$, $\alpha'_k = \alpha_k \geq \beta_1$ for all $1 < k < j$, $\alpha'_k \geq \alpha'_{k+1}$ for all $k \geq j$, and
\[
T' = \sum^n_{k=1} \alpha'_k P_k \in \cconv(\U(T)).
\]
Furthermore, if $Q = \sum^n_{k=2} P_k$, then $QSQ \prec_{\tau_Q}  QT'Q$ in $Q\fA Q$.
\end{lem}
\begin{proof}
Note $j \geq 2$ by Example \ref{exam:finite-spectrum-eigenvalue-functions} along with the fact that $S \prec_\tau T$.  In addition, note $\alpha_1 > \alpha_j$.

Consider 
\[
T_0 = \alpha_1 P_1 + \alpha_j P_j \in (P_1 + P_j) \fA (P_1 + P_j).
\]
If $\beta_1 \in [\tau_{P_1 + P_j}(T_0), \alpha_1]$, write $\beta_1 = t \alpha_1 + (1-t) \tau_{P_1 + P_j}(T_0)$ with $t \in [0,1]$ and let
\[
\alpha'_1 = \beta_1, \quad \alpha'_j = t \alpha_j + (1-t) \tau_{P_1 + P_j}(T_0), \qand \alpha'_k = \alpha_k \text{ for all }k \neq 1, j.
\]
Otherwise, if $\beta_1 \notin [\tau_{P_1 + P_j}(T_0), \alpha_1]$, let 
\[
\alpha'_1 = \alpha'_j = \tau_{P_1 + P_j}(T_0), \qand \alpha'_k = \alpha_k \text{ for all }k \neq 1, j.
\]
Notice, in this later case, that $\alpha'_1 = \alpha'_j > \beta_1$.  Furthermore, in both cases,
\[
\alpha'_1 \tau(P_1) + \alpha'_j \tau(P_j) = \alpha_1 \tau(P_1) + \alpha_j \tau(P_j).
\]

If $T' = \sum^n_{k=1} \alpha'_k P_k$, then by applying Lemma \ref{lem:pinching} to $T_0\in (P_1 + P_j) \fA (P_1 + P_j)$ and by appealing to a direct sum argument, we obtain $T' \in \cconv(\U(T))$.

We claim that $S \prec_{\tau} T'$.  For each $k \in \{0, 1,\ldots, n\}$, let $s_k = \sum^k_{j=1} \tau(P_j)$.  Notice $s_k < s_{k+1}$ for all $k$, $s_0 = 0$, $s_n = 1$, and $\lambda^\tau_{T}(s) = \alpha_k$ and $\lambda^\tau_{S}(s) = \beta_k$ for all $s \in [s_{k-1}, s_k)$ by Example \ref{exam:finite-spectrum-eigenvalue-functions}.   Notice, in both of the above cases, that $\alpha'_k \geq \beta_1$ for all $k < j$ and $\alpha'_k \geq \alpha'_{k+1}$ for all $k \geq j$ (as $\alpha'_j \geq \alpha_j$).  Therefore, Definition \ref{defn:majorization} and Example \ref{exam:finite-spectrum-eigenvalue-functions} imply that $\lambda^\tau_{T'}(s) = \alpha'_k = \lambda^\tau_{T}(s)$ for all $s \in [s_{k-1}, s_k)$ with $k > j$,  
\[
\int^{s_j}_0 \lambda^\tau_{T'}(s) \, ds = \sum^j_{k=1} \alpha'_k \tau(P_k) = \sum^j_{k=1} \alpha_k \tau(P_k) = \int^{s_j}_0 \lambda^\tau_{T}(s) \, ds,
\]
and $\lambda^\tau_{T'}(s) \geq \beta_1$ for all $s < s_{j-1}$.
Consequently, if $t \in [0, s_{j-1}]$, we see that
\[
\int^t_0 \lambda^\tau_{T'}(s) - \lambda^\tau_S(s) \, ds \geq \int^t_0 \beta_1 - \beta_1 \, ds = 0.
\]

For $t \in [s_{j-1}, s_j)$, we will need to divide the proof into two cases.  First, if $\alpha'_j \geq \beta_j$, then $\alpha'_k \geq \beta_j$ for all $k < j$.  Consequently $\lambda^\tau_{T'}(s) \geq \beta_j$ on $[0, s_{j})$  so
\begin{align*}
\int^t_0 \lambda^\tau_{T'}(s) - \lambda^\tau_S(s) \, ds &=\int^{s_{j-1}}_0 \lambda^\tau_{T'}(s) - \lambda^\tau_S(s) \, ds + \int^t_{s_{j-1}} \lambda^\tau_{T'}(s) - \beta_j \, ds\geq 0 + 0.
\end{align*}
Otherwise suppose $\alpha'_j < \beta_j$. Notice $\alpha'_k < \alpha'_j < \beta_j \leq \beta_1 \leq \alpha'_l$ for all $k \geq j$ and $l < j$.  Thus
\[
\int^{s_{j-1}}_0 \lambda^\tau_{T'}(s)   \, ds = \sum^{j-1}_{k=1} \alpha'_k \tau(P_k)
\]
and $\lambda^{\tau}_{T'}(s) = \alpha'_j$ for all $s \in [s_{j-1}, s_j)$.  Consequently
\begin{align*}
\int^t_0 \lambda^\tau_{T'}(s) - \lambda^\tau_S(s) \, ds &= \sum^{j-1}_{k=1} (\alpha'_k - \beta_k) \tau(P_k) + \int^t_{s_{j-1}} \alpha'_j - \beta_j \, ds \\
& \geq \sum^{j-1}_{k=1} (\alpha'_k - \beta_k) \tau(P_k) + \int^{s_j}_{s_{j-1}} \alpha'_j - \beta_j \, ds \\
&= \sum^j_{k=1} (\alpha'_k - \beta_k) \tau(P_k)\\
&= \sum^j_{k=1} (\alpha_k - \beta_k) \tau(P_k) = \int^{s_j}_{0} \lambda^\tau_{T}(s) - \lambda^\tau_S(s) \, ds \geq 0.
\end{align*}

Finally, if $t \geq s_j$, then
\begin{align*}
\int^t_0 \lambda^\tau_{T'}(s) - \lambda^\tau_S(s) \, ds &= \int^{s_j}_0 \lambda^\tau_{T'}(s) - \lambda^\tau_S(s) \, ds + \int^t_{s_{j}} \lambda^\tau_{T'}(s) - \lambda^\tau_S(s) \, ds \\
&= \int^{s_j}_0 \lambda^\tau_{T}(s) - \lambda^\tau_S(s) \, ds + \int^t_{s_{j}} \lambda^\tau_{T}(s) - \lambda^\tau_S(s) \, ds \geq 0
\end{align*}
with equality when $t = 1$.  Thus the proof that $S \prec_\tau T'$ is complete.

Postponing the discussion of the $\alpha'_1 \neq \beta_1$ case, we demonstrate that if $\alpha'_1 = \beta_1$ then $QSQ \prec_{\tau_Q}  QT'Q$ in $Q\fA Q$.  For each $k \in \{1,\ldots, n\}$, let $s'_k = \sum^k_{j=2} \tau_Q(P_j)$.  Notice $s'_k < s'_{k+1}$ for all $k$, $s'_1 = 0$, $s'_n = 1$, and $\lambda^{\tau_Q}_{QSQ}(s) = \beta_k$ for all $s \in [s'_{k-1}, s'_k)$ by Example \ref{exam:finite-spectrum-eigenvalue-functions}.   In the case $\alpha'_1 = \beta_1$, we note that $\alpha'_j \leq \beta_1 \leq \alpha'_l$ for all $l < j$, and $\alpha'_k \geq \alpha'_{k+1}$ for all $k \geq j$.  Consequently,  $\lambda^{\tau_Q}_{QT'Q}(s) \geq \beta_1$ for all $s < s'_{j-1}$, $\lambda^{\tau_Q}_{T'}(s) = \alpha'_k$ for all $s \in [s'_{k-1}, s'_k)$ with $k \geq j$, and
\[
\int^{s'_{j-1}}_0 \lambda^{\tau_Q}_{Q T' Q}(s) \, ds = \sum^{j-1}_{k=2} \alpha'_k \tau_Q(P_k).
\]
Moreover, one can verify that
\[
 \lambda^{\tau_Q}_{QT'Q}\left(\frac{s - \tau(P_1)}{\tau(Q)}  \right) = \lambda^{\tau}_{T'}(s) \qqand  \lambda^{\tau_Q}_{QSQ}\left(\frac{s - \tau(P_1)}{\tau(Q)}  \right)= \lambda^{\tau}_{S}(s)
\]
for all $s \geq s_j$.

If $t < s'_{j-1}$, then
\[
\int^t_0 \lambda^{\tau_Q}_{QT'Q}(s) - \lambda^{\tau_Q}_{QSQ}(s) \, ds \geq \int^t_0 \beta_1 - \beta_2 \, ds \geq 0.
\]
If $t \in [s'_{j-1}, s'_j]$, we see that 
\begin{align*}
&\int^t_0 \lambda^{\tau_Q}_{QT'Q}(s) - \lambda^{\tau_Q}_{QSQ}(s) \, ds \\
&= \int^{s'_{j-1}}_0 \lambda^{\tau_Q}_{QT'Q}(s) - \lambda^{\tau_Q}_{QSQ}(s) \, ds + \int^t_{s'_{j-1}} \lambda^{\tau_Q}_{QT'Q}(s) - \lambda^{\tau_Q}_{QSQ}(s) \, ds\\
&= \frac{1}{\tau(Q)} \sum^{j-1}_{k=2} (\alpha'_k - \beta_k) \tau(P_k) + \int^t_{s'_{j-1}} \alpha'_j - \beta_j \, ds \\
&= \frac{1}{\tau(Q)} \sum^{j-1}_{k=1} (\alpha'_k - \beta_k) \tau(P_k) + \int^t_{s'_{j-1}} \alpha'_j - \beta_j \, ds \\
&= \frac{1}{\tau(Q)} \int^{s_{j-1}}_0 \lambda^{\tau}_{T'}(s) - \lambda^{\tau}_{S}(s) \, ds + \int^t_{s'_{j-1}} \alpha'_j - \beta_j \, ds .
\end{align*}
In particular, for $t = s'_j$, we see that
\begin{align*}
\int^{s'_j}_0 \lambda^{\tau_Q}_{QT'Q}(s) - \lambda^{\tau_Q}_{QSQ}(s) \, ds &= \frac{1}{\tau(Q)} \int^{s_{j-1}}_0 \lambda^{\tau}_{T'}(s) - \lambda^{\tau}_{S}(s) \, ds + \int^{s'_j}_{s'_{j-1}} \alpha'_j - \beta_j \, ds \\
&= \frac{1}{\tau(Q)} \int^{s_{j-1}}_0 \lambda^{\tau}_{T'}(s) - \lambda^{\tau}_{S}(s) \, ds + (\alpha'_j - \beta_j) \frac{\tau(P_j)}{\tau(Q)} \\
&= \frac{1}{\tau(Q)} \int^{s_{j}}_0 \lambda^{\tau}_{T'}(s) - \lambda^{\tau}_{S}(s) \, ds.
\end{align*}
If $\alpha'_j \geq \beta_j$, then
\[
\int^t_0 \lambda^{\tau_Q}_{QT'Q}(s) - \lambda^{\tau_Q}_{QSQ}(s) \, ds \geq \frac{1}{\tau(Q)} \int^{s_{j-1}}_0 \lambda^{\tau}_{T'}(s) - \lambda^{\tau}_{S}(s) \, ds \geq 0
\]
for all $t \in [s'_{j-1}, s'_j]$.  Otherwise $\alpha'_j < \beta_j$ and
\begin{align*}
\int^t_0 \lambda^{\tau_Q}_{QT'Q}(s) - \lambda^{\tau_Q}_{QSQ}(s) \, ds &\geq \int^{s'_j}_0 \lambda^{\tau_Q}_{QT'Q}(s) - \lambda^{\tau_Q}_{QSQ}(s) \, ds \\
&= \frac{1}{\tau(Q)} \int^{s_{j}}_0 \lambda^{\tau}_{T'}(s) - \lambda^{\tau}_{S}(s) \, ds \geq 0.
\end{align*}
Finally, if $t > s'_j$, 
\begin{align*}
&\int^t_0 \lambda^{\tau_Q}_{QT'Q}(s) - \lambda^{\tau_Q}_{QSQ}(s) \, ds \\
&= \int^{s'_{j}}_0 \lambda^{\tau_Q}_{QT'Q}(s) - \lambda^{\tau_Q}_{QSQ}(s) \, ds + \int^t_{s'_j} \lambda^{\tau_Q}_{QT'Q}(s) - \lambda^{\tau_Q}_{QSQ}(s) \, ds\\
&= \frac{1}{\tau(Q)} \int^{s_{j}}_0 \lambda^{\tau}_{T'}(s) - \lambda^{\tau}_{S}(s) \, ds 
\\ & \quad + \frac{1}{\tau(Q)}\int^{\tau(Q)t + \tau(P_1)}_{s_j} \lambda^{\tau_Q}_{QT'Q}\left(\frac{s - \tau(P_1)}{\tau(Q)}  \right) - \lambda^{\tau_Q}_{QSQ}\left(\frac{s - \tau(P_1)}{\tau(Q)}  \right) \, ds \\
&= \frac{1}{\tau(Q)} \int^{s_{j}}_0 \lambda^{\tau}_{T'}(s) - \lambda^{\tau}_{S}(s) \, ds  + \frac{1}{\tau(Q)} \int^{\tau(Q)t + \tau(P_1)}_{s_j} \lambda^{\tau}_{T'}(s) - \lambda^{\tau}_{S}(s) \, ds  \geq 0
\end{align*}
with equality to zero when $t = 1$.  Hence $QSQ \prec_{\tau_Q}  QT'Q$ in $Q\fA Q$.

To complete the proof, we notice the proof is complete when $\beta_1 \in [\tau_{P_1 + P_j}(T_0), \alpha_1]$ (i.e. the $\alpha'_1 = \beta_1$ case).  Otherwise, repeat the above proof with $j$ replaced with $j+1$ and $T$ replaced with $T'$.  Note we end up obtaining that $\alpha'_j \geq \alpha'_{j+1}$ under this recursion as the first iteration yields $\alpha'_1 = \alpha'_j$ and the second iteration would average $\alpha'_1$ with $\alpha'_{j+1} \leq \alpha_j < \alpha'_j$ to yield $\alpha''_k$ with $\alpha''_{j} = \alpha'_j > \alpha''_{j+1}$.  This process must eventually obtain $\alpha'_1 = \beta_1$ by reaching the case that $\beta_1 \in [\tau_{P_1 + P_j}(T_0), \alpha_1]$ for if we must apply the proof with $j = n$ and we produce a self-adjoint operator $T'$ with $S \prec_\tau T'$, $\alpha'_1 > \beta_1$, and $\alpha'_k \geq \beta_1 \geq \beta_l$ for all $k$ and $l$, we have a contradiction to the fact that $S \prec_\tau T'$ (which guarantees $\tau(S) = \tau(T')$).  Furthermore, note we obtain $QSQ \prec_{\tau_Q} QT'T$ at the last step of this iterative process.
\end{proof}

\begin{proof}[Proof of Theorem \ref{thm:classification-of-convex-hull}]
Let $T \in \fA$ be self-adjoint.  Note the inclusion
\[
\cconv(\U(T)) \subseteq \{S \in \fA \, \mid \, S^* = S, S \prec_\tau T\}
\]
follows by Lemma \ref{lem:easy-direction-for-convex-hulls}.

To prove the other inclusion, let $S \in \fA$ be self-adjoint with $S \prec_\tau T$.  By Lemma \ref{lem:reduction-to-finite-spectrum}, we may assume without loss of generality that $S$ and $T$ have finite spectrum. 

 Let $\{P_k\}^n_{k=1}$, $\{Q_k\}^n_{k=1}$, $\{\alpha_k\}^n_{k=1}$, and $\{\beta_k\}^n_{k=1}$ be as in Lemma \ref{lem:getting-equal-trace-projections} so that
\[
T = \sum^n_{k=1} \alpha_k Q_k \qqand S = \sum^n_{k=1} \beta_k P_k.
\]
Since $\fA$ has strong comparison of projections, there exists a unitary $U \in \fA$ such that $U^*Q_kU = P_k$ for all $k$.  Hence $U^*TU = \sum^n_{k=1} \alpha_k P_k$.  Since $\lambda^\tau_{U^*TU}(s) = \lambda^\tau_{T}(s)$ for all $s \in [0,1)$, $S \prec_\tau U^*TU$.  Consequently, Example \ref{exam:finite-spectrum-eigenvalue-functions} and Definition \ref{defn:majorization} implies $\alpha_1 \geq \beta_1 \geq \beta_n \geq \alpha_n$.  

If $\alpha_1 = \alpha_n$, then $T = S = \tau(T) I_\fA$ and there is nothing to prove.  Otherwise, we may apply Lemma \ref{lem:reduction-argument} to obtain, for some $\{\alpha'_k\}^n_{k=2} \subseteq \bR$,  that
\[
T' = \beta_1 P_1 + \sum^n_{k=2} \alpha'_k P_k \in\cconv(\U(U^*TU))
\qand
QSQ \prec_{\tau_Q} QT'Q \text{ in } Q \fA Q,
\]
where $Q = \sum^n_{k=2} P_k$.  In addition, note Lemma \ref{lem:reduction-argument} produces $\{\alpha'_k\}^n_{k=2}$ so that $QSQ$ and $QT'Q$ in $Q \fA Q$ are either equal or satisfy the hypotheses of Lemma \ref{lem:reduction-argument}; that is, $QSQ \prec_{\tau_Q} QT'Q$, $\alpha'_{k+1} \leq \alpha'_k$ for all $k \geq j$, $\alpha'_k = \alpha_k \geq \beta_1 \geq \beta_2$ for all $1 < k < j$, and, if $j = 2$, $\alpha'_2 \geq \beta_2 \geq \beta_n \geq \alpha'_n$ by Definition \ref{defn:majorization} and Example \ref{exam:finite-spectrum-eigenvalue-functions}.  Therefore, by applying Lemma \ref{lem:reduction-argument} at most another $n-1$ times, we obtain that
\[
S \in \cconv(\U(U^*TU))= \cconv(\U(T)). \qedhere
\]
\end{proof}

\section{Classification of Additional Sets}
\label{sec:Other-Majorizations}

In this section, we will study additional sets based on eigenvalue and singular value functions in C$^*$-algebras satisfying the hypotheses of Theorem \ref{thm:classification-of-convex-hull}.  We begin by studying the distance between unitary orbits of self-adjoint operators.  The following result is the main result of \cite{ST1992}.  We provide a different (but very similar) proof using the technology of this paper.

\begin{thm}[see \cite{ST1992}]
\label{thm:distance-between-unitary-orbits}
Let $\fA$ be a unital C$^*$-algebra with real rank zero that has strong comparison of projections with respect to a faithful tracial state $\tau$.   If $S, T \in \fA$ are self-adjoint, then
\[
\dist(\U(S), \U(T)) = \sup\{ |\lambda^\tau_S(s) - \lambda^\tau_T(s)| \, \mid \, s \in [0,1)\}.
\]
In particular, $S$ and $T$ are approximately unitarily equivalent if and only if $\lambda^\tau_S(s) = \lambda^\tau_T(s)$ for all $s \in [0,1)$ if and only if $S \prec_\tau T$ and $T \prec_\tau S$.
\end{thm}
\begin{proof}
By parts (\ref{part:difference}, \ref{part:unitary}) of Theorem \ref{thm:properties-of-eigenvalue-functions}, we have
\[
|\lambda^\tau_S(s) - \lambda^\tau_T(s)| = |\lambda^\tau_{U^*SU}(s) - \lambda^\tau_{V^*TV}(s)| \leq \left\|U^*SU - V^*TV\right\|
\]
for all unitaries $U, V \in \fA$ and $s \in [0,1)$.  Hence 
\[
\sup\{ |\lambda^\tau_S(s) - \lambda^\tau_T(s)| \, \mid \, s \in [0,1)\} \leq \dist(\U(S), \U(T)).
\]

For the other inclusion, fix $\epsilon > 0$.  Since $\fA$ has real rank zero, there exists self-adjoint operators $S', T' \in \fA$ with finite spectrum such that 
\[
\left\|S - S'\right\| < \epsilon \qqand \left\|T - T'\right\| < \epsilon.
\]
Note
\[
|\lambda^\tau_{T}(s) - \lambda^\tau_{T'}(s)| < \epsilon \qqand |\lambda^\tau_S(s) - \lambda^\tau_{S'}(s)| < \epsilon
\]
for all $s \in [0,1)$ by part (\ref{part:difference}) of Theorem \ref{thm:properties-of-eigenvalue-functions}.

Let $\{P_k\}^n_{k=1}$, $\{Q_k\}^n_{k=1}$, $\{\alpha_k\}^n_{k=1}$, and $\{\beta_k\}^n_{k=1}$ be as in  Lemma \ref{lem:getting-equal-trace-projections} so that
\[
T' = \sum^n_{k=1} \alpha_k P_k \qqand S' = \sum^n_{k=1} \beta_k Q_k.
\]
If $s_k = \sum^k_{j=1} \tau(Q_j)$ for all $k \in \{0, 1, \ldots, n\}$, Example \ref{exam:finite-spectrum-eigenvalue-functions} implies $\lambda^\tau_{T'}(s)= \alpha_k$ and $\lambda^\tau_{S'}(s) = \beta_k$ for all $s \in [s_{k-1}, s_k)$.  Furthermore, since $\tau(P_k) = \tau(Q_k)$ for all $k$ and since $\fA$ has strong comparison of projections, there exists a unitary $U \in \fA$ such that $U^*P_kU = Q_k$ for all $k$ and, consequently, $U^*T'U  = \sum^n_{k=1} \alpha_k Q_k$.  Hence
\begin{align*}
\left\|U^*TU - S\right\| &\leq 2\epsilon +  \left\|U^*T'U - S'\right\|\\
&= 2\epsilon + \sup\{|\alpha_k - \beta_k| \, \mid \, k \in \{1,\ldots, n\}\} \\
&= 2\epsilon + \sup\{|\lambda^\tau_{T'}(s) - \lambda^\tau_{S'}(s)| \, \mid \, s \in [0,1)\} \\
&\leq 4\epsilon + \sup\{|\lambda^\tau_{T}(s) - \lambda^\tau_{S}(s)| \, \mid \, s \in [0,1)\}.
\end{align*}
As $\epsilon > 0$ was arbitrary, the proof is complete.
\end{proof}

The following result is an adaptation of \cite{HN1991}*{Theorem 2.4}.
\begin{thm}
\label{thm:distance-to-convex-hull}
Let $\fA$ be a unital C$^*$-algebra with real rank zero that has strong comparison of projections with respect to a faithful tracial state $\tau$.   If $S, T \in \fA$ are self-adjoint, then
\[
\dist(S, \conv(\U(T))) = \sup_{ t \in (0, 1)} \frac{1}{t} \max\left\{ \int^t_0 \lambda^\tau_{S}(s) - \lambda^\tau_T(s) \, ds, \int^1_{1-t} \lambda^\tau_{T}(s) - \lambda^\tau_S(s) \, ds    \right\}.
\]
\end{thm}
\begin{proof}
Let $\alpha$ be the quantity on the right-hand side of the desired equation.  Suppose $T' \in \conv(\U(T))$.  Then $T' \prec_\tau T$ by Lemma \ref{lem:easy-direction-for-convex-hulls}.  Consequently, by part (\ref{part:difference}) of Theorem \ref{thm:properties-of-eigenvalue-functions} and by Definition \ref{defn:majorization}, 
\begin{align*}
\left\|T' - S\right\| &\geq \frac{1}{t} \int^t_0 \lambda^\tau_{S}(s) - \lambda^\tau_{T'}(s) \, ds \geq \frac{1}{t} \int^t_0 \lambda^\tau_{S}(s) - \lambda^\tau_{T}(s) \, ds \text{ and}\\
\left\|T' - S\right\| &\geq \frac{1}{t} \int^1_{1-t} \lambda^\tau_{T'}(s) - \lambda^\tau_{S}(s) \, ds \geq \frac{1}{t} \int^1_{1-t} \lambda^\tau_{T}(s) - \lambda^\tau_{S}(s) \, ds.
\end{align*}
Therefore $\dist(S, \conv(\U(T))) \geq \alpha$.

For the other inequality, first suppose $\alpha \leq 0$. Then
\[
\int^t_0 \lambda^\tau_{S}(s) - \lambda^\tau_T(s) \, ds \leq 0 \qqand \int^1_{1-t} \lambda^\tau_{T}(s) - \lambda^\tau_S(s) \, ds \leq 0
\]
for all $t \in (0, 1)$.  The first inequality implies 
\[
\int^t_0 \lambda^\tau_{S}(s) \, ds\leq \int^t_0 \lambda^\tau_T(s) \, ds
\]
for all $t \in [0,1]$, and by letting $t$ tend to $1$, the second inequality then implies
\[
\int^1_0 \lambda^\tau_{S}(s) \,ds = \int^1_0 \lambda^\tau_T(s) \, ds.
\]
Consequently, $\alpha = 0$ and $S \prec_\tau T$.  Thus Theorem \ref{thm:classification-of-convex-hull} implies $S \in \cconv(\U(T))$ so equality is obtained in this case.

Otherwise, suppose $\alpha > 0$.  Let $\epsilon > 0$.  Since $\fA$ has real rank zero, there exists self-adjoint operators $S', T' \in \fA$ with finite spectrum such that 
\[
\left\|S - S'\right\| < \epsilon \qqand \left\|T - T'\right\| < \epsilon.
\]
In addition, by part (\ref{part:difference}) of Theorem \ref{thm:properties-of-eigenvalue-functions},
\[
|\lambda^\tau_{S}(s) - \lambda^\tau_{S'}(s)| < \epsilon \qqand |\lambda^\tau_{T}(s) - \lambda^\tau_{T'}(s)| < \epsilon
\]
for all $s \in [0, 1)$.  By the definition of $\alpha$, we obtain
\begin{align*}
\int^t_0 \lambda^\tau_{S'}(s) - \alpha - 2\epsilon \, ds &\leq \int^t_0 \lambda^\tau_{S}(s) - \alpha - \epsilon \, ds \leq \int^t_0 \lambda^\tau_T(s) -\epsilon \, ds \leq \int^t_0 \lambda^\tau_{T'}(s) \, ds\\
\int^1_t \lambda^\tau_{S'}(s) + \alpha + 2\epsilon \, ds &\geq \int^t_0\lambda^\tau_{S}(s) + \alpha + \epsilon \, ds \geq \int^t_0\lambda^\tau_{T}(s) + \epsilon \, ds \geq \int^t_0\lambda^\tau_{T'}(s) \, ds
\end{align*}
for all $t \in (0, 1)$.  Consequently, using non-increasing rearrangements and \cite{HN1991}*{Proposition 1.4(1)} applied to $f_1(s) = \lambda^\tau_{S'}(s) - \alpha - 2\epsilon$, $f_2(s) = \lambda^\tau_{S'}(s) + \alpha + 2\epsilon$, and $g(s) = \lambda^\tau_{T'}(s)$, there exists a real-valued, non-increasing function $h \in L_\infty[0,1]$ such that
\begin{equation}
f_1(s) \leq h(s) \leq f_2(s) \label{eqn:bounds-on-h}
\end{equation}
for all $s \in [0,1)$ and $h \prec \lambda^\tau_{T'}$.

Let $\{P_k\}^n_{k=1}$, $\{Q_k\}^n_{k=1}$, $\{\alpha_k\}^n_{k=1}$, and $\{\beta_k\}^n_{k=1}$ be as in  Lemma \ref{lem:getting-equal-trace-projections} so that
\[
T' = \sum^n_{k=1} \alpha_k P_k \qqand S' = \sum^n_{k=1} \beta_k Q_k.
\]
Furthermore, for $k \in \{0, 1, \ldots, n\}$, let $s_k = \sum^k_{j=1} \tau(Q_k)$, let
\[
\alpha'_k = \frac{1}{s_k - s_{k-1}} \int^{s_k}_{s_{k-1}} h(s) \, ds,
\]
and let $T_0 = \sum^n_{k=1} \alpha'_k P_k$.  Notice $\alpha'_k \geq \alpha'_{k+1}$ for all $k$ as $h$ is non-increasing.  Since $h \prec \lambda^\tau_{T'}$, Examples \ref{exam:finite-spectrum-eigenvalue-functions} and \ref{exam:averaging} imply that $T_0 \prec_\tau T$.  Hence Theorem \ref{thm:classification-of-convex-hull} implies $T_0 \in \cconv(\U(T'))$.  

Since $\fA$ has strong comparison of projections, there exists a unitary $U \in \fA$ such that $U^*P_kU = Q_k$ for all $k$.  Therefore $U^*T_0U = \sum^n_{k=1} \alpha'_k Q_k$.  However, due to the definition of $\alpha'_k$,  equation (\ref{eqn:bounds-on-h}), and  Example \ref{exam:finite-spectrum-eigenvalue-functions}, we see that
\[
\left\|U^*T_0U - S'\right\| \leq \alpha + 2\epsilon.
\]
Therefore, since $U^*T_0U \in \cconv(\U(T'))$, $\left\|T - T'\right\| < \epsilon$, and $\left\|S - S'\right\| < \epsilon$, we obtain that
\[
\dist(S, \conv(\U(T))) \leq \alpha + 4\epsilon
\]
thereby completing the proof.
\end{proof}

Since tracial states are norm continuous, Theorem \ref{thm:distance-to-convex-hull} immediately implies the following.

\begin{cor}
Let $\fA$ be a unital C$^*$-algebra with real rank zero that has strong comparison of projections with respect to a faithful tracial state $\tau$.   If $S, T \in \fA$ are self-adjoint, then
\[
\dist(\conv(\U(S)), \conv(\U(T))) = |\tau(S) - \tau(T)|.
\]
\end{cor}

Using the technology of Section \ref{sec:Preliminaries}, we are also able to study arbitrary operators based on their singular value functions.  The following object will play the role of the singular value decomposition of matrices for infinite dimensional C$^*$-algebras.
\begin{defn}
For a unital C$^*$-algebra $\fA$ and an element $T \in \fA$, the \emph{closed two-sided unitary orbit} of $T$ is
\[
\N(T) = \overline{ \{UTV \, \mid \, U, V \text{ unitaries in }\fA\}  }.
\]
\end{defn}

Our goal is to classify closed two-sided unitary orbits using singular values.  We restrict to C$^*$-algebras with stable rank one as the following well-known lemma directly implies every operator almost has a polar decomposition.
\begin{lem}
\label{lem:close-absolute-values}
Let $\fA$ be a unital C$^*$-algebra and let $M, \epsilon > 0$. There exists a $0 < \delta < \epsilon$ such that if $A, B \in \fA$, $\left\|A\right\| \leq M$, and $\left\|A - B\right\| < \delta$, then $\left\| |A| - |B|\right\| < \epsilon$.
\end{lem}
\begin{cor}
\label{cor:almost-polar-decompositions}
Let $\fA$ be a unital C$^*$-algebra with stable rank one and let $T \in \fA$.  Then for all $\epsilon > 0$ there exists a unitary $U \in \fA$ such that $\left\|T - U|T|\right\| < \epsilon$.
\end{cor}
\begin{proof}
Use Lemma \ref{lem:close-absolute-values} along with the fact that invertible elements in unital C$^*$-algebras have polar decompositions.
\end{proof}
\begin{lem}
\label{lem:pointwise-convergence-of-singular-values}
Let $\fA$ be a unital C$^*$-algebra with a faithful tracial state $\tau$.  If $(T_n)_{n\geq 1} \subseteq \fA$ converges in norm to $T \in \fA$, then $\mu^\tau_T(s) = \lim_{n\to \infty} \mu^\tau_{T_n}(s)$ for all $s \in [0,1)$.
\end{lem}
\begin{proof}
Recall $\mu^\tau_S(s) = \lambda^\tau_{|S|}(s)$ for all $S \in \fA$.  Since $T = \lim_{n\to \infty} T_n$, we obtain $|T| = \lim_{n\to \infty} |T_n|$ by Lemma \ref{lem:close-absolute-values}.  The result then follows by part (\ref{part:difference}) of Theorem \ref{thm:properties-of-eigenvalue-functions}.
\end{proof}

The following is a generalization of \cite{KS2015}*{Theorem 2.11} to C$^*$-algebras.
\begin{prop}
\label{prop:closed-two-sided-unitary-orbit}
Let $\fA$ be a unital C$^*$-algebra with real rank zero, stable rank one, and strong comparison of projections with respect to a faithful tracial state $\tau$.   If $S, T \in \fA$, then $S \in \N(T)$ if and only if $\mu^\tau_S(s) = \mu^\tau_T(s)$ for all $s \in [0,1)$.
\end{prop}
\begin{proof}
If $U, V \in \fA$ are unitaries, then 
\[
\mu^\tau_{UTV}(s) = \lambda^\tau_{|UTV|}(s) = \lambda^\tau_{V^*|T|V}(s) = \lambda^\tau_{|T|}(s) = \mu^\tau_T(s)
\]
for all $s \in [0,1)$ by part (\ref{part:unitary}) of Theorem \ref{thm:properties-of-eigenvalue-functions}.  Consequently, if $S \in \N(T)$, then $\mu^\tau_S(s) = \mu^\tau_T(s)$ for all $s \in [0,1)$ by Lemma \ref{lem:pointwise-convergence-of-singular-values}.

For the converse direction, suppose $\mu^\tau_S(s) = \mu^\tau_T(s)$ for all $s \in [0,1)$ and let $\epsilon > 0$.  By Corollary \ref{cor:almost-polar-decompositions}, there exists unitaries $U, V \in \fA$ such that
\[
\left\|T - U|T|\right\| < \epsilon \qqand \left\|S - V|S|\right\| < \epsilon.
\]
Furthermore, since
\[
\lambda^\tau_{|T|}(s) = \mu^\tau_T(s) = \mu^\tau_S(s) = \lambda^\tau_{|S|}(s)
\]
for all $s \in [0,1)$, Theorem \ref{thm:distance-between-unitary-orbits} implies there exists a unitary $W \in \fA$ such that $\left\|W^*|T|W - |S|\right\| < \epsilon$. Hence
\[
\left\|VW^*U^*TW - S\right\| \leq 2\epsilon + \left\|VW^*|T'|W - V|S'|\right\| < 3\epsilon.
\]
Since $\epsilon > 0$ was arbitrary, the proof is complete.
\end{proof}

Our next results provide descriptions of all operators whose eigenvalue (singular value) function is dominated by another operators eigenvalue (singular value) function.  In particular, these notions of majorization are related to Cuntz equivalence, but are significantly stronger (i.e. requiring bounded sequences for approximations).  The following result is a generalization of \cite{H1987}*{Theorems 3.1}.

\begin{prop}
\label{prop:compression-closure}
Let $\fA$ be a unital C$^*$-algebra with real rank zero that has strong comparison of projections with respect to a faithful tracial state $\tau$.   If $S, T \in \fA$ are positive operators, then
\[
S \in \overline{ \{A^*TA \, \mid \, A \in \fA, \left\|A \right\| \leq 1\}  }
\]
if and only if $\lambda^\tau_S(s) \leq \lambda^\tau_T(s)$ for all $s \in [0,1)$.
\end{prop}
\begin{proof}
If $A \in \fA$ is such that $\left\|A\right\| \leq 1$, then
\[
\lambda^\tau_{A^*TA}(s) \leq \left\|A\right\|^2 \lambda^\tau_T(s) \leq \lambda^\tau_T(s)
\]
for all $s \in [0,1)$ by part (\ref{part:conjugation}) of Theorem \ref{thm:properties-of-eigenvalue-functions}.  Consequently, one direction follows from part (\ref{part:difference}) of Theorem \ref{thm:properties-of-eigenvalue-functions}.

For the other direction, suppose $\lambda^\tau_S(s) \leq \lambda^\tau_T(s)$ for all $s \in [0,1)$.  Let $\epsilon > 0$.  By Lemma \ref{lem:reduction-to-finite-spectrum} there exists positive operators $S', T' \in \fA$ with finite spectra such that $\left\|T - T'\right\| < \epsilon$, $\left\|S - S'\right\| < \epsilon$, and $\lambda^\tau_{S'}(s) \leq \lambda^\tau_{T'}(s)$ for all $s \in [0,1)$.  Let $\{P_k\}^n_{k=1}$, $\{Q_k\}^n_{k=1}$, $\{\alpha_k\}^n_{k=1}$, and $\{\beta_k\}^n_{k=1}$ be as in Lemma \ref{lem:getting-equal-trace-projections} so that
\[
T' = \sum^n_{k=1} \alpha_k P_k \qqand S' = \sum^n_{k=1} \beta_k Q_k.
\]
Since $T', S' \geq 0$, $\alpha_k, \beta_k \geq 0$ for all $k$.  Furthermore, Example \ref{exam:finite-spectrum-eigenvalue-functions} along with the fact that $\lambda^\tau_{S'}(s) \leq \lambda^\tau_{T'}(s)$ for all $s \in [0,1)$ implies $\beta_k \leq \alpha_k$ for all $k$.

Since $\fA$ has strong comparison of projections, there exists a unitary $U \in \fA$ such that $U^*P_kU = Q_k$ for all $k$ so that $U^*T'U = \sum^n_{k=1} \alpha_k Q_k$.  For each $k$, let
\[
\gamma_k = \left\{
\begin{array}{ll}
\sqrt{\frac{\beta_k}{\alpha_k}} & \mbox{if } \beta_k \neq 0 \\
0 & \mbox{if } \beta_k = 0 
\end{array} \right. .
\]
Consequently, if $A = \sum^n_{k=1} \gamma_k Q_k \in \fA$, then $\left\|A\right\| \leq 1$ and $A^*U^*T'UA = S'$.  Hence
\[
\left\|A^*U^*TUA - S\right\| \leq 2\epsilon + \left\|A^*U^*T'UA - S'\right\| = 2\epsilon.
\]
As $\epsilon > 0$, the result follows.
\end{proof}

\begin{prop}
Let $\fA$ be a unital C$^*$-algebra with real rank zero, stable rank one, and strong comparison of projections with respect to a faithful tracial state $\tau$.   If $S, T \in \fA$, then
\[
S \in \overline{ \{ATB \, \mid \, A, B \in \fA, \left\|A\right\|, \left\|B \right\| \leq 1\}  }
\]
if and only if $\mu^\tau_S(s) \leq \mu^\tau_T(s)$ for all $s \in [0,1)$.
\end{prop}
\begin{proof}
If $A, B \in \fA$ are such that $\left\|A\right\|, \left\|B\right\| \leq 1$, then
\[
\mu^\tau_{ATB}(s) \leq \left\|A\right\|\left\|B\right\| \mu^\tau_T(s) \leq \mu^\tau_T(s)
\]
for all $s \in [0,1)$ by part (\ref{part:mu-contractive}) of Theorem \ref{thm:properties-of-singular-value-functions}.  Consequently, one direction follows from Lemma \ref{lem:pointwise-convergence-of-singular-values}.

For the other direction, suppose $\mu^\tau_S(s) \leq \mu^\tau_T(s)$ for all $s \in [0,1)$. Consequently $\lambda^\tau_{|S|}(s) \leq \lambda^\tau_{|T|}(s)$ for all $s \in [0,1)$.  Thus Proposition \ref{prop:compression-closure} implies for all $\epsilon > 0$ there exists an $A \in \fA$ with $\left\|A\right\| \leq 1$ such that $\left\||S| - A^*|T|A\right\| < \epsilon$.  Furthermore, Corollary \ref{cor:almost-polar-decompositions} implies there exists unitaries $U, V \in \fA$ such that $\left\|S - V|S|\right\| < \epsilon$ and $\left\|T - U|T|\right\| < \epsilon$.  Thus
\begin{align*}
\left\|S - VA^*U^*TA\right\| \leq \left\|S - VA^*|T|A\right\| + \epsilon \leq \left\|S - V|S|\right\| + 2\epsilon \leq 3\epsilon.
\end{align*}
The result follows.
\end{proof}

%

To complete this section, we desire to analyze the notion of (absolute) submajorization as defined in Definition \ref{defn:absolute-submajorization}.  In particular, we desire an analogue of \cite{H1987}*{Theorem 2.5(2)} for C$^*$-algebras.  The following useful lemma shows if one positive operator submajorizes an operator, then conjugating by a specific contractive operator almost yields majorization.
\begin{lem}
\label{lem:getting-majorization-from-absolute-majorization}
Let $\fA$ be a unital C$^*$-algebra with real rank zero and strong comparison of projections with respect to a faithful tracial state $\tau$.   If $S, T \in \fA$ are positive operators such that $S \prec^w_\tau T$, then for all $\epsilon > 0$ there exists positive operators $S', T' \in \fA$ and  an $A \in \fA$ with $\left\|A\right\| \leq 1$ such that 
\[
\left\|S - S'\right\| \leq \epsilon, \quad \left\|T - T'\right\| \leq \epsilon, \qand S' \prec_\tau A^*T'A.
\]
\end{lem}
\begin{proof}
Fix $\epsilon > 0$.  By Lemma \ref{lem:reduction-to-finite-spectrum} there exists positive operators $S', T' \in \fA$ with finite spectra such that 
\[
\left\|S - S'\right\| \leq \epsilon, \quad \left\|T - T'\right\| \leq \epsilon, \qand S' \prec^w_\tau T'.
\]
Let $\{P_k\}^n_{k=1}$, $\{Q_k\}^n_{k=1}$, $\{\alpha_k\}^n_{k=1}$, and $\{\beta_k\}^n_{k=1}$ be as in Lemma \ref{lem:getting-equal-trace-projections} so that
\[
T' = \sum^n_{k=1} \alpha_k P_k \qqand S' = \sum^n_{k=1} \beta_k Q_k.
\]
For each $k \in \{0, 1,\ldots, n\}$, let $s_k = \sum^n_{k=1} \tau(P_k)$.  

Consider the function $f : [0,1] \to \bR$ defined by
\[
f(t) = \int^t_0 \lambda^\tau_{T'}(s) \, ds - \int^1_0 \lambda^\tau_{S'}(s) \, ds.
\]
Since $f$ is continuous, $f(0) \leq 0$, and $f(1) \geq 0$, there exists a $t_0 \in [0,1]$ such that $f(t_0) = 0$.  Let $t' = \sup\{t \in [0,1] \, \mid \, f(t) = 0\}$ and choose $k' \in \{1, \ldots, n\}$ such that $t' \in [s_{k'-1}, s_{k'})$ (with $k' = n$ if $t' =1$).  Notice this implies
\[
\int^{s_{k'-1}}_0 \lambda^\tau_{T'}(s) \, ds \leq \int^1_0 \lambda^\tau_{S'}(s) \, ds \leq \int^{s_{k'}}_0 \lambda^\tau_{T'}(s) \, ds.
\]
Choose $q \in [0,1]$ such that
\[
\int^1_0 \lambda^\tau_{S'}(s) \, ds = \int^{s_{k'-1}}_0 \lambda^\tau_{T'}(s) \, ds + q \int^{s_{k'}}_{s_{k'-1}} \lambda^\tau_{T'}(s) \, ds
\]
and let $A = q P_{k'} + \sum^{k'-1}_{k=1} P_k$.  Clearly $\left\|A\right\| \leq 1$ and
\[
A^*T'A = q \alpha_{k'} P_{k'}+  \sum^{k'-1}_{k=1} \alpha_k P_k.
\]

We claim that $S' \prec_\tau A^*T'A$.  By Example \ref{exam:finite-spectrum-eigenvalue-functions}, we know $\lambda^\tau_{S'}(s) = \beta_k$, $\lambda^\tau_{T'}(s) = \alpha_k$ for all $s \in [s_{k-1}, s_k)$, $\lambda^\tau_{A^*T'A}(s) = \alpha_k$ for all $s \in [s_{k-1}, s_k)$ with $k < k'$, $\lambda^\tau_{A^*T'A}(s) = q \alpha_k$ for all $s \in [s_{k'-1}, s_{k'})$, and $\lambda^\tau_{A^*TA} = 0$ for all $s \geq s_{k'}$.  Consequently, for all $t \in [0, s_{k'-1})$,
\[
\int^t_0 \lambda^\tau_{A^*T'A}(s) - \lambda^\tau_{S'}(s) \, ds = \int^t_0 \lambda^\tau_{T'}(s) - \lambda^\tau_{S'}(s) \, ds \geq 0.
\]

If $t \in [s_{k'-1}, s_{k'})$, then
\[
\int^t_0 \lambda^\tau_{A^*T'A}(s) - \lambda^\tau_{S'}(s) \, ds = \int^{s_{k'-1}}_0 \lambda^\tau_{T'}(s) - \lambda^\tau_{S'}(s) \, ds + (t - s_{k'-1})(q \alpha_{k'} - \beta_{k'}).
\]
If $q \alpha_{k'} \geq \beta_{k'}$, then 
\[
\int^t_0 \lambda^\tau_{A^*T'A}(s) - \lambda^\tau_{S'}(s) \, ds  \geq \int^{s_{k'-1}}_0 \lambda^\tau_{T'}(s) - \lambda^\tau_{S'}(s) \, ds \geq 0.
\]
Otherwise $q \alpha_{k'} < \beta_{k'}$ and
\begin{align*}
& \int^t_0 \lambda^\tau_{A^*T'A}(s) - \lambda^\tau_{S'}(s) \, ds \\
&\geq \int^{s_{k'-1}}_0 \lambda^\tau_{T'}(s) - \lambda^\tau_{S'}(s) \, ds + (s_{k'} - s_{k'-1})(q \alpha_{k'} - \beta_{k'}) \\
&= \int^{s_{k'-1}}_0 \lambda^\tau_{T'}(s) - \lambda^\tau_{S'}(s) \, ds +  q \int^{s_{k'}}_{s_{k'-1}} \lambda^\tau_{T'}(s) \, ds - \int^{s_{k'}}_{s_{k'-1}}\lambda^\tau_{S'}(s) \, ds \\
&= \int^{1}_{s_{k'}}\lambda^\tau_{S'}(s) \, ds \geq0
\end{align*}
as $\lambda^\tau_{S'}(s) \geq 0$ for all $s$ as $S' \geq 0$.  

Finally, if $t \geq s_{k'}$, then
\begin{align*}
& \int^t_0 \lambda^\tau_{A^*T'A}(s) - \lambda^\tau_{S'}(s) \, ds \\
&=  \int^{s_{k'-1}}_0 \lambda^\tau_{T'}(s) \, ds + q \int^{s_{k'}}_{s_{k'-1}} \lambda^\tau_{T'}(s) \, ds  -    \int^t_0 \lambda^\tau_{S'}(s) \, ds \\
&=  \int^1_0 \lambda^\tau_{S'}(s) \, ds  -    \int^t_0 \lambda^\tau_{S'}(s) \, ds \geq 0
\end{align*}
with equality when $t = 1$ as $\lambda^\tau_{S'}(s) \geq 0$ for all $s$ as $S' \geq 0$.  Hence $S' \prec_\tau A^*T'A$.
\end{proof}

\begin{prop}
\label{prop:weak-majorization-characterization}
Let $\fA$ be a unital C$^*$-algebra with real rank zero and strong comparison of projections with respect to a faithful tracial state $\tau$.   If $S, T \in \fA$ are positive operators, then
\[
S \in  \cconv\left(\{A^*TA \, \mid \, A \in \fA, \left\|A\right\|\leq 1\}\right)  
\]
if and only if $S \prec^w_\tau T$.
\end{prop}
\begin{proof}
If $\{A_k\}^n_{k=1} \subseteq \fA$ are such that $\left\|A_k\right\| \leq 1$ for all $k$, $\{t_k\}^n_{k=1} \subseteq [0,1]$ are such that $\sum^n_{k=1} t_k = 1$, and $S' = \sum^n_{k=1} t_k A_k^* T A_k$, then $S' \geq 0$ and
\begin{align*}
\int^t_0 \lambda^\tau_{S'}(s) \, ds &\leq \int^t_0 \sum^n_{k=1} t_k \left\|A_k\right\|^2 \lambda^\tau_{T}(s) \, ds \leq \int^t_0 \lambda^\tau_{T}(s) \, ds 
\end{align*}
by parts (\ref{part:positive-multiple}, \ref{part:conjugation}, \ref{part:concave-integral}) of Theorem \ref{thm:properties-of-eigenvalue-functions}.  Thus one inclusion follows from part (\ref{part:difference}) of Theorem \ref{thm:properties-of-eigenvalue-functions}.

For the other direction, suppose $S \prec^w_\tau T$.  Let $\epsilon > 0$.  By Lemma \ref{lem:getting-majorization-from-absolute-majorization} there exists positive operators $S', T' \in \fA$ and  an $A \in \fA$ with $\left\|A\right\| \leq 1$ such that 
\[
\left\|S - S'\right\| \leq \epsilon, \quad \left\|T - T'\right\| \leq \epsilon, \qand S' \prec_\tau A^*T'A.
\]
As
\[
S' \in \cconv(\U(A^*T'A))
\]
by Theorem \ref{thm:classification-of-convex-hull}, the result follows.
\end{proof}

\begin{prop}
Let $\fA$ be a unital C$^*$-algebra with real rank zero, stable rank one, and strong comparison of projections with respect to a faithful tracial state $\tau$.   If $S, T \in \fA$, then
\[
S \in  \cconv\left(\{ATB \, \mid \, A, B \in \fA, \left\|A\right\|, \left\|B\right\| \leq 1\}\right)  
\]
if and only if $S \prec^w_\tau T$.
\end{prop}
\begin{proof}
If $\{A_k\}^n_{k=1}, \{B_k\}^n_{k=1} \subseteq \fA$ are such that $\left\|A_k\right\|, \left\|B_k\right\| \leq 1$ for all $k$, $\{t_k\}^n_{k=1} \subseteq [0,1]$ are such that $\sum^n_{k=1} t_k = 1$, and $S' = \sum^n_{k=1} t_k A_k T B_k$, then
\begin{align*}
\int^t_0 \mu^\tau_{S'}(s) \, ds \leq \int^t_0 \sum^n_{k=1} t_k \left\|A_k\right\|\left\|B_k\right\| \mu^\tau_{T}(s) \, ds \leq \int^t_0 \mu^\tau_{T}(s) \, ds 
\end{align*}
by parts (\ref{part:mu-scalar}, \ref{part:mu-contractive},  \ref{part:mu-concave-integral}) of Theorem \ref{thm:properties-of-eigenvalue-functions}.  Thus one inclusion follows from Lemma \ref{lem:pointwise-convergence-of-singular-values}.

For the other direction, suppose $S \prec^w_\tau T$.  Thus $|S| \prec^w_\tau |T|$ so Proposition \ref{prop:weak-majorization-characterization} implies 
\[
|S| \in  \cconv\left(\{A^*|T|A \, \mid \, A \in \fA, \left\|A\right\|\leq 1\}\right) .
\]
The result then follows by approximation arguments along with Lemma \ref{lem:close-absolute-values}.
\end{proof}

\section{Purely Infinite C$^*$-Algebras}
\label{sec:Purely-Infinite}

In this section, we will prove the following result describing the closed convex hulls of unitary orbits of self-adjoint operators $T$ in unital, simple, purely infinite C$^*$-algebras based on the spectrum of $T$, denoted $\sigma(T)$.
\begin{thm}
\label{thm:convex-hull-of-unitary-orbit-in-uspi}
Let $\fA$ be a unital, simple, purely infinite C$^*$-algebra and let $T \in \fA$ be self-adjoint.  Then
\[
\cconv(\U(T)) = \{S \in \fA \, \mid \, S^* = S, \sigma(S) \subseteq \conv(\sigma(T))\}.
\]
\end{thm}

\begin{rem}
\label{rem:reduction-to-finite-spectrum}
Before proceeding, we briefly outline the approach to the proof, beginning with the following simplifications.  
Note the inclusion
\[
\cconv(\U(T)) \subseteq  \{S \in \fA \, \mid \, S^* = S, \sigma(S) \subseteq \conv(\sigma(T))\}
\]
follows from the facts that elements of $\conv(\U(T))$ are self-adjoint when $T$ is self-adjoint, and, if $\alpha I_\fA \leq T \leq \beta I_\fA$, then $\alpha I_\fA \leq S \leq \beta I_\fA$ for all $S \in \conv(\U(T))$.  

Since unital, simple, purely infinite C$^*$-algebras have real rank zero by \cite{Z1990}, to verify the reverse inclusion it suffices to consider self-adjoint $S, T \in \fA$ with finite spectrum and $\sigma(S) \subseteq \conv(\sigma(T))$ by the continuous functional calculus.  Furthermore, note this problem is invariant under simultaneous multiplying the operators by non-zero real numbers and simultaneous translation of the operators by a real constant.  As such, it suffices to prove the result for positive $T$ with $\left\|T\right\| = 1$ and $0, 1 \in \sigma(T)$.

We will demonstrate it suffices to prove the result when $T$ is a projection.  As in Section \ref{sec:Convex-Hulls-Of-Unitary-Orbits}, this will be done by constructing (possibly non-unital) embeddings of arbitrarily larger matrix algebras into $\fA$ and using Theorem \ref{thm:majorization-in-factors}.  Subsequently, we will verify that the result holds for $T$ a projection and $S \in \bC I_\fA$, again appealing to Theorem \ref{thm:majorization-in-factors}.  The result will follow for arbitrary $S$ with finite spectrum by an application of K-Theory.  
\end{rem}

We begin with the following well-known result for purely infinite C$^*$-algebras.
\begin{lem}
\label{lem:direct-sum-of-arbitrary-projection}
Let $\fA$ be a unital, simple, purely infinite C$^*$-algebra and let $P, Q \in \fA$ be orthogonal non-zero projection.  For any $n \in \bN$ there exists a collection $\{P_k\}^n_{k=1}$ of pairwise orthogonal subprojections of $P$ such that each $P_k$ is Murray-von Neumann equivalent to $Q$.
\end{lem}

%

By `a non-trivial projection', we mean a non-zero projection $P$ with $P \neq I_\fA$.

\begin{lem}
\label{lem:two-point-spectrum-to-one-point}
Let $\fA$ be a unital, simple, purely infinite C$^*$-algebra and let $P \in \fA$ be a non-trivial projection.  If $\alpha, \beta \in \bR$ and $T = \alpha P + \beta(I_\fA - P)$, then $\alpha I_\fA \in \cconv(\U(T))$.
\end{lem}
\begin{proof}
Clearly the result holds if $\alpha = \beta$ so suppose $\alpha \neq \beta$.  Using Remark \ref{rem:reduction-to-finite-spectrum}, by scaling and translating, we may assume that $\alpha = 1$ and $\beta = 0$.

Let $n \in \bN$ be arbitrary.  By Lemma \ref{lem:direct-sum-of-arbitrary-projection} there exists a collection $\{P_k\}^n_{k=1}$ of pairwise orthogonal subprojections of $P$ such that $P_k \sim I_\fA - P$ for all $k$.  Using the partial isometries implementing the equivalence of $\{I_\fA - P\} \cup \{P_k\}^n_{k=1}$, a copy of $\M_{n+1}(\bC)$ may be constructed in $\fA$ such that the unit of $\M_{n+1}(\bC)$ is $P'_n := I_\fA - P + \sum^n_{k=1} P_k$ and $T$ may be viewed as the operator
\[
T = \diag(0, 1, \ldots, 1) \oplus (I_\fA - P'_n) \in \M_{n+1}(\bC) \oplus (I_\fA - P'_n) \fA (I_\fA - P'_n) \subseteq \fA.
\]
Since any self-adjoint matrix majorizes its trace (see Lemma \ref{lem:easy-direction-for-convex-hulls}), we obtain by Theorem \ref{thm:majorization-in-factors} that
\[
\frac{n}{n+1} I_{n+1} \in \cconv(\U(\diag(0, 1, \ldots, 1)))
\]
where the unitary orbit is computed in $\M_{n+1}(\bC)$.  Thus, by a direct sum argument, we obtain
\[
\frac{n}{n+1} P'_n + (I_\fA - P'_n) \in \conv(\U(T)).
\]
By taking the limit as $n \to \infty$, we obtain $I_\fA \in \cconv(\U(T))$.
\end{proof}

\begin{lem}
\label{lem:suffices-to-consider-two-point-spectrum-operators}
Let $\fA$ be a unital, simple, purely infinite C$^*$-algebra and let $\{P_k\}^n_{k=1}$ be a collection of pairwise orthogonal, non-zero projections.  If $T = \sum^n_{k=1} \lambda_k P_k$ for some real numbers $\{\lambda_k\}^n_{k=1} \in \bR$, then 
\[
\lambda_1 \left(\sum^{n-1}_{k=1} P_k\right) + \lambda_n P_n \in \cconv(\U(T)).
\]
\end{lem}
\begin{proof}
The result follows by using Lemma \ref{lem:two-point-spectrum-to-one-point} recursively on compressions of $\fA$ (which remain unital, simple, purely infinite C$^*$-algebras).
\end{proof}

\begin{lem}
\label{lem:getting-three-point-spectra-from-projection}
Let $\fA$ be a unital, simple, purely infinite C$^*$-algebra and let $P \in \fA$ be a non-trivial projection.  For each $\gamma \in [0,1] \cap \bQ$, there exists pairwise orthogonal, non-zero projections $Q_1, Q_2, Q_3$ such that $Q_1 + Q_2 + Q_3 = I_\fA$ and
\[
0 Q_1 + \gamma Q_2 + 1 Q_3 \in \cconv(\U(P)).
\]
\end{lem}
\begin{proof}
Note the cases $\gamma = 0, 1$ are trivial.  Otherwise, fix $n \in \bN$ and choose $k \in \{1,\ldots, n-1\}$ so that $\gamma = \frac{k}{n}$.  Let $Q \in \fA$ be any non-trivial projection.  By Lemma \ref{lem:direct-sum-of-arbitrary-projection} there exists a collection $\{P_j\}^{k+1}_{j=1}$ of pairwise orthogonal subprojections of $P$ such that $P_j \sim Q$ for all $j$. Similarly there exists a collection $\{P'_j\}^{n-k+1}_{j=1}$ of pairwise orthogonal subprojections of $I_\fA - P$ such that $P'_j \sim Q$ for all $j$. 

Let
\begin{align*}
Q_1 = (I_\fA - P) - \sum^{n-k}_{j=1} P'_j,\quad
Q_2 = \sum^{k}_{j=1} P_j + \sum^{n-k}_{j=1} P'_j, \qand
Q_3 = P - \sum^{k}_{j=1} P_j.
\end{align*}
Since $P_{k+1} \leq Q_3$ and $P_{n-k+1} \leq Q_1$, it is clear that $Q_1$, $Q_2$, and $Q_3$ are pairwise orthogonal, non-zero projections such that $Q_1 + Q_2 + Q_3 = I_\fA$.  Using the partial isometries implementing the equivalence of $\{P_j\}^{k}_{j=1} \cup \{P'_j\}^{n-k}_{j=1}$, a copy of $\M_{n}(\bC)$ can be constructed in $\fA$ such that the unit of $\M_{n}(\bC)$ is $Q_2$ and
\[
P = 0Q_1 \oplus D \oplus 1Q_3 \in  Q_1 \fA Q_1 \oplus \M_{n}(\bC) \oplus Q_3 \fA Q_3 \subseteq \fA
\]
where $D$ is a diagonal matrix with $1$ appearing along the diagonal exactly $k$ times and $0$ appearing along the diagonal exactly $n-k$ times.  
Since any self-adjoint matrix majorizes its trace (see Lemma \ref{lem:easy-direction-for-convex-hulls}), we obtain by Theorem \ref{thm:majorization-in-factors} and a direct sum argument that
\[
0Q_1 + \gamma Q_2 + 1 Q_3 \in \conv(\U(P)). \qedhere
\]
\end{proof}

\begin{lem}
\label{lem:any-scalar-for-a-projection}
Let $\fA$ be a unital, simple, purely infinite C$^*$-algebra and let $P \in \fA$ be a non-trivial projection.  For each $\gamma \in [0,1]$, $\gamma I_\fA \in \cconv(\U(P))$.
\end{lem}
\begin{proof}
By applying approximations, it suffices to prove the result for $\gamma \in (0, 1) \cap \bQ$.  By Lemma \ref{lem:getting-three-point-spectra-from-projection} there exists pairwise orthogonal, non-zero projections $Q_1, Q_2, Q_3$ such that $Q_1 + Q_2 + Q_3 = I_\fA$ and
\[
0 Q_1 + \gamma Q_2 + 1 Q_3 \in \cconv(\U(P)).
\]

Choose two non-zero subprojections $Q'_1$ and $Q'_3$ of $Q_2$ such that $Q'_1 + Q'_3 = Q_2$.  By applying Lemma \ref{lem:two-point-spectrum-to-one-point} to $0 Q_1 + \gamma Q'_1 \in (Q_1 + Q'_1) \fA(Q_1 + Q'_1)$, we obtain that 
\[
\gamma (Q_1 + Q'_1) \in  \cconv(\U(0 Q_1 + \gamma Q'_1))
\]
(where the quantity on the right-hand side is computed in $ (Q_1 + Q'_1) \fA(Q_1 + Q'_1)$).  Similarly
\[
\gamma (Q_3 + Q'_3) \in  \cconv(\U(1 Q_3 + \gamma Q'_3))
\]
Hence, by the fact that $0 Q_1 + \gamma Q_2 + 1 Q_3$ is a direct sum of $0 Q_1 + \gamma Q'_1$ and $1 Q_3 + \gamma Q'_3$, we obtain that
\[
\gamma I_\fA = \gamma (Q_1 + Q'_1) + \gamma (Q_3 + Q'_3) \in \cconv(\U(P)). \qedhere
\]
\end{proof}

\begin{proof}[Proof of Theorem \ref{thm:convex-hull-of-unitary-orbit-in-uspi}]
By Remark \ref{rem:reduction-to-finite-spectrum}, we may assume $\sigma(S)$ and $\sigma(T)$ are finite so that there exists $\{\lambda_j\}^m_{j=1}, \{\alpha_k\}^n_{k=1} \subseteq \bR$ with $\lambda_k < \lambda_{k+1}$ for all $k$ and $\alpha_k \in \conv(\{\lambda_j\}^m_{j=1})$ for all $k$, and two collections of pairwise orthogonal non-zero projections $\{P_j\}^m_{j=1}$ and $\{Q_k\}^n_{k=1}$ with $\sum^m_{j=1} P_j = I_\fA = \sum^n_{k=1} Q_k$ such that
\[
T = \sum^m_{j=1} \lambda_j P_j \qqand S = \sum^n_{k=1} \alpha_k Q_k.
\]

The result is trivial if $m = 1$ so we assume $m \geq 2$.  Furthermore, by translation and scaling, it suffices to prove the result when $\lambda_1 = 0$ and $\lambda_m = 1$.  Furthermore, by Lemma \ref{lem:suffices-to-consider-two-point-spectrum-operators}, we may assume that $m = 2$.  For simplicity, let $P = P_m$ so $P_1 = I_\fA - P$ and $T = P$.

Since $\fA$ is a unital, simple, purely infinite C$^*$-algebra, there exists a collection $\{P'_k\}^{n-1}_{k=1}$ of non-zero subprojections of $P$ and  a collection $\{P''_k\}^{n-1}_{k=1}$ of non-zero subprojections of $I_\fA - P$ such that $P'_k + P''_k \sim Q_k$, $P'_n = P - \sum^{n-1}_{k=1} P'_k$ is non-zero, and $P''_n = \sum^{n-1}_{k=1} P''_k$ is non-zero.  For each $k \in \{1,\ldots, n\}$, let $Q'_k = P'_k + P''_k$.  Therefore
\[
\sum^n_{k=1} [Q_k]_0 = [I_\fA]_0 = \sum^n_{k=1} [Q'_k]_0 = [Q'_n]_0 + \sum^{n-1}_{k=1} [Q_k]_0.
\]
Hence $[Q_n]_0 = [Q'_n]_0$ so $Q_n \sim Q'_n$ by \cite{C1978}*{Theorem 1.4}.  

Notice
\[
T = \oplus^n_{k=1} (1 P'_k + 0 P''_k) \in \bigoplus^n_{k=1} Q'_k \fA Q'_k.
\]
Since $P'_k$ and $P''_k$ are non-zero for each $k$ and since $Q'_k \fA Q'_k$ is a unital, simple, purely infinite C$^*$-algebra, by applying Lemma \ref{lem:any-scalar-for-a-projection} in each $Q'_k \fA Q'_k$ and by taking a direct sum, we obtain
\[
\sum^n_{k=1} \alpha_k Q'_k \in \cconv(\U(T)).
\]
Since $\sum^n_{k=1} \alpha_k Q'_k$ is unitarily equivalent to $S$ by the fact that $Q_k \sim Q'_k$ for all $k$, we obtain that $S \in \cconv(\U(T))$.
\end{proof}

We note the following adaptation of \cite{HN1991}*{Theorem 4.2}.
\begin{cor}
Let $\fA$ be a unital, simple, purely infinite C$^*$-algebra.  If $S, T \in \fA$ are self-adjoint, then
\[
\dist(S, \conv(\U(T))) = \sup_{x \in \sigma(S)} \dist(x, \conv(\sigma(T))).
\]
\end{cor}
\begin{proof}
First, suppose $T' \in \conv(\U(T))$.  Let $\pi : \fA \to \B(\H)$ be a faithful representation of $\fA$ (whose existence is guaranteed by the GNS construction). By \cite{H1967}*{Problem 171}, for every self-adjoint operator $A \in \B(\H)$,
\[
\conv(\sigma(A)) = \overline{\{\langle A \eta, \eta \rangle \, \mid \, \eta \in \H, \left\|\eta\right\| = 1\}   }.
\]
Let $\eta \in \H$ be such that $\left\|\eta\right\| =1$.  Since
\[
\left\|T' - S\right\| \geq |\langle \pi(T' -S)\eta, \eta\rangle| \geq \dist(\langle \pi(S)\eta, \eta\rangle, \conv(\sigma(T))),
\]
we obtain that
\[
\dist(S, \conv(\U(T))) \geq \sup_{x \in \sigma(S)} \dist(x, \conv(\sigma(T))).
\]

For the reverse inclusion, defined a continuous function $f : \bR \to \bR$ so that $f(x) \in \conv(\sigma(T))$ for all $x$ and
\[
|x - f(x)| = \dist(x, \conv(\sigma(T)))
\]
for all $x \in \bR$.  Let $T' = f(S)$.  Therefore,  by the continuous functional calculus,  $\sigma(T') = f(\sigma(S)) \subseteq \conv(\sigma(T))$.  Hence $T' \in \cconv(\sigma(T))$ by Theorem \ref{thm:convex-hull-of-unitary-orbit-in-uspi}.  Since
\[
\left\|S - T'\right\| = \sup_{x \in \sigma(S)} \left\|x - f(x)\right\| = \sup_{x \in \sigma(S)} \dist(x, \conv(\sigma(T))),
\]
the reverse inclusion holds.
\end{proof}

To conclude this paper, we note the proof of Theorem \ref{thm:convex-hull-of-unitary-orbit-in-uspi} can be improved to normal operators provided $K_1(\fA)$ is trivial or, more generally by \cite{C1981}, for normal operators $N$ such that $\lambda I_\fA - N$ is an element of the connected component containing $I_\fA$ in the set of invertible elements of $\fA$, denoted $\fA^{-1}_0$, for all $\lambda \notin \sigma(N)$.  This is a generalization of \cite{HN1991}*{Theorem 4.1} and we only sketch the modifications to the proof.

\begin{thm}
\label{thm:convex-hull-of-unitary-orbit-in-uspi-normal}
Let $\fA$ be a unital, simple, purely infinite C$^*$-algebra and let $N_1, N_2 \in \fA$ be normal operators with $\lambda I_\fA - N_k \in \fA^{-1}_0$ for all $\lambda \notin \sigma(N_k)$ and for all $k$.  Then $N_2 \in \cconv(\U(N_1))$ if and only if $\sigma(N_2) \subseteq \conv(\sigma(N_1))$.
\end{thm}
\begin{proof}
Suppose $N_2 \in \cconv(\U(N_1))$.  Let $(M_n)_{n\geq 1} \subseteq \conv(\U(N_1))$ be such that $N_2 = \lim_{n \to \infty} M_n$ and let $\pi : \fA \to \B(\H)$ be a faithful representation of $\fA$. By \cite{H1967}*{Problem 171}, for every normal operator $A \in \B(\H)$,
\[
\conv(\sigma(A)) = \overline{\{\langle A \eta, \eta \rangle \, \mid \, \eta \in \H, \left\|\eta\right\| = 1\}   }.
\]
Since $M_n \in \conv(\U(N_1))$, we obtain $\langle\pi(M_n) \eta, \eta \rangle \in \conv(\sigma(N_1))$ for all $\eta \in \H$ with $\left\|\eta \right\| =1$.  Therefore, since $\langle \pi(N_2) \eta, \eta \rangle = \lim_{n\to \infty} \langle \pi(M_n)\eta, \eta\rangle$, we obtain $\sigma(N_2) \subseteq \conv(\sigma(N_1))$.

For the converse direction, note by \cites{L1996} that $N_1$ and $N_2$ can be approximated by normal operators with finite spectra.  Thus, by an application of the continuous functional calculus, it suffices to prove that if $\sigma(N_2)$ and $\sigma(N_1)$ are finite and $\sigma(N_2) \subseteq \conv(\sigma(N_1))$, then $N_2 \in \cconv(\U(N_1))$.  Furthermore, by using similar direct sum arguments as in the proof of Theorem \ref{thm:convex-hull-of-unitary-orbit-in-uspi}, it suffices to prove the result in the case that $N_2 \in \bC I_\fA$.

Note that Lemma \ref{lem:two-point-spectrum-to-one-point} holds when $\alpha$ and $\beta$ are complex numbers by applying rotations and translations.  Hence by applying the same ideas as in Lemma \ref{lem:suffices-to-consider-two-point-spectrum-operators}, we may reduce to the case that $N$ has exactly three points in its spectrum.  

Suppose $\sigma(N_1) = \{\alpha_1, \alpha_2, \alpha_3\}$ and $\gamma \in \conv(\sigma(N_1))$.  Then there exists a permutation $\sigma$ on $\{1,2,3\}$ and $t, r \in [0,1]$ such that if $\gamma' = t \alpha_{\pi(1)} + (1-t)\alpha_{\pi(2)}$ then $\gamma = r \gamma' + (1-r) \alpha_{\pi(3)}$.  Consequently, by applying rotations, translations, compressions, and Lemma \ref{lem:any-scalar-for-a-projection} first with the spectral projections corresponding to $\alpha_{\pi(1)}$ and $\alpha_{\pi(2)}$, and then again with the result and the spectral projection corresponding to $\alpha_{\pi(3)}$, the result is obtained.
\end{proof}

\section*{Acknowledgements}

The author would like to thank Stuart White for his aid with Examples \ref{exam:properties1}, \ref{exam:properties3}, David Kerr for his aid with Example \ref{exam:properties4}, and Ken Dykema for his aid with Example \ref{exam:properties5}.  The author would also like to thank Ping Ng for a conversation that led to Remark \ref{rem:short-proof-of-simple}.


\begin{thebibliography}{13}



\bib{A1989}{article}{
    author={Ando, T.},
    title={Majorization, doubly stochastic matrices, and comparison of eigenvalues},
    journal={Linear Algebra Appl.},
    volume={118},
    date={1989},
    pages={163-248}
}

\bib{AM2008}{article}{
    author={Argerami, M.},
    author={Massey, P.},
    title={The local form of doubly stochastic maps and joint majorization in II$_1$ factors},
	journal={Integral Equations Operator Theory},
	volume={61},
	number={1},
    year={2008},
    pages={1-19}
}


\bib{AK2006}{book}{
    author={Arveson, W.},
    author={Kadison, V.},
    title={Diagonals of self-adjoint operators},
    series={Operator Theory, Operator Algebras, and Applications},
    publisher={Amer. Math. Soc.},
    volume={414},
    date={2006},
    place={Providence, RI},
    pages={247-263}
}



\bib{A1982}{article}{
    author={Avitzour, D.},
    title={Free products of C$^*$-algebras},
    journal={Trans. Amer. Math. Soc.},
    volume={271},
    year={1982},
    pages={423-465}
}


\bib{AD1984}{article}{
    author={Azoff, E.},
    author={Davis, C.},
    title={On distances between unitary orbits of self-adjoint operators},
    journal={ACTA SCI MATH},
    volume={47},
    number={3-4},
    year={1984},
    pages={419-439}
}








\bib{B1946}{article}{
    author={Birkhoff, G.},
    title={Tres observaciones sobre el algebra lineal},
    journal={Univ. Nac. Tucumán Rev. Ser. A},
    volume={5},
    date={1946},
    pages={147-151}
}



\bib{B1988}{book}{
    author={Blackadar, Bruce},
    title={Comparison theory for simple C$^*$-algebras},
    publisher={Cambridge Univ. Press},
    series={in Operator Algebras and Applications, LMS Lecture Notes Series},
    volume={135},
    place={Cambridge-New York},
    date={1988},
    pages={21-54}
}




\bib{BP1991}{article}{
    author={Brown, L.},
    author={Pedersen, G.},
    title={C$^*$-algebras of real rank zero},
    journal={J. Funct. Anal.},
    volume={99},
    number={1},
    date={1991},
    pages={131-149}
}









\bib{C2013}{thesis}{
    author={Cheong, C. W.},
    title={Distances between approximate unitary equivalence classes of self-adjoints in C$^*$-algebras},
    place={Purdue},
	type={Ph.D. thesis},
    date={2013},
}




\bib{C1974}{article}{
    author={Chong, Kong Ming},
    title={Some extensions of a theorem of Hardy, Littlewood and P\'{o}lya and their applications},
    journal={Canad. J. Math},
    volume={26},
    date={1974},
    pages={1321-1340}
}


\bib{C1976}{article}{
    author={Chong, Kong Ming},
    title={Doubly stochastic operators and rearrangement theorems},
    journal={J. Math. Anal. Appl.},
    volume={56},
     number={2},
    date={1976},
    pages={309-316}
}






\bib{C1978}{article}{
    author={Cuntz, J.},
    title={Dimension functions on simple C$^*$-algebras},
    journal={Math. Ann.},
    volume={233},
    number={2},
    date={1978},
    pages={145-153}
}




\bib{C1981}{article}{
    author={Cuntz, J.},
    title={K-theory for certain C$^*$-algebras},
    journal={Ann. Math.},
    volume={131},
    date={1981},
    pages={181-197}
}



\bib{D1995}{article}{
    author={Dadarlat, M.},
    title={Approximately unitarily equivalent morphisms and inductive limit C$^*$-algebras},
    journal={K-Theory},
    volume={9},
    year={1995},
    pages={117-137}
}





\bib{D1984}{article}{
    author={Davidson, K. R.},
    title={The distance between unitary orbits of normal elements in the Calkin algebra},
    journal={P. Roy. Soc. Edinb.},
    volume={99A},
    year={1984},
    pages={35-43}
}



\bib{D1986}{article}{
    author={Davidson, K. R.},
    title={The distance between unitary orbits of normal operators},
    journal={Acta Sci. Math.},
    volume={50},
    year={1986},
    pages={213-223}
}



\bib{D1988}{article}{
    author={Davidson, K. R.},
    title={Estimating the Distance Between Unitary Orbits},
    journal={J. Operator Theory},
    volume={20},
    year={1988},
    pages={21-40}
}

\bib{D1969}{book}{
    author={Dixmier, L.},
    title={Les alg\`{e}bres d'op\'{e}rateurs dans l'espace Hilbertien},
    edition={2},
    date={1969},
    place={Gauthier-Villars Paris},
}



\bib{DHR1997}{article}{
    author={Dykema, Kenneth J.},
    author={Haagerup, U.},
    author={R{\o}rdam, M.},
    title={The stable rank of some free product C$^*$-algebras},
    journal={Duke Math. J.},
    volume={90},
    date={1997},
    pages={95-121, correction in \textbf{94} (1998), 213}
}




\bib{DR2000}{article}{
    author={Dykema, Kenneth J.},
    author={R{\o}rdam, M.},
    title={Projections in free product C$^*$–algebras, II},
    journal={Math. Z.},
    volume={234},
    number={1},
    date={2000},
    pages={103-113}
}




\bib{DS2015}{article}{
    author={Dykema, K. J.},
    author={Skoufranis, P.},
    title={Numerical ranges in II$_1$ factors},
    eprint={arXiv:1503.05766},
    date={2015},
    pages={29}
}








\bib{E1993}{article}{
    author={Elliott, George A.},
    title={On the classification of C$^*$-algebras of real rank zero},
	journal={J. reine angew. Math},
	volume={443},
	number={35},
    year={1993},
    pages={179-219}
}

\bib{F1982}{article}{
    author={Fack, T.},
    title={Sur la notion de valuer caract\'{e}ristique},
    journal={J. Operator Theory},
    volume={7},
    number={2},
    date={1982},
    pages={207-333}
}



\bib{FK1986}{article}{
    author={Fack, T.},
    author={Kosaki, H.},
    title={Generalized $s$-numbers of $\tau$-measurable operators},
    journal={Pacific J. Math},
    volume={123},
    number={2},
    date={1986},
    pages={269-300}
}




\bib{GS1977}{article}{
    author={Goldberg, M.},
    author={Straus, E.},
    title={Elementary inclusion relations for generalized numerical ranges},
	journal={Linear Algebra Appl.},
	volume={18},
	number={1},
    year={1977},
    pages={1-24}
}


\bib{HZ1984}{article}{
  author={Haagerup, Uffe},
  author={Zsid\'{o}, L.},
  title={Sur la propri\'{e}t\'{e} de Dixmier pour les C$^*$-alg\`{e}bres},
  year={1984},
  journal={C. R. Acad. Sc. Paries},
  volume={298},
  number={8},
  pages={173-176}
}



\bib{H1967}{book}{
   author={Halmos, P. R.},
   title={A Hilbert Space Problem Book},
	volume={1},
   publisher={Princeton: van Nostrand},
   date={1967},
}








\bib{HLP1929}{article}{
  author={Hardy, G. H.},
  author={Littlewood, J. E.},
  author={P\'{o}lya, G.},
  title={Some simple inequalities satisfied by convex functions},
  year={1929},
  journal={Messenger Math},
  number={58},
  pages={145-152}
}



\bib{HLP1952}{book}{
  title={Inequalities},
  author={Hardy, G. H.},
  author={Littlewood, J. E.},
  author={P\'{o}lya, G.},
  year={1952},
  edition={2},
  publisher={Cambridge Univ. Press},
  place={London/New York}
}

\bib{H1987}{article}{
    author={Hiai, F.},
    title={Majorization and stochastic maps in von Neumann algebras},
    journal={J. Math. Anal. Appl.},
    volume={127},
    date={1987},
    pages={18-48}
}





\bib{HN1989}{article}{
   author={Hiai, Fumio},
   author={Nakamura, Yoshihiro},
   title={Distance between unitary orbits in von Neumann algebras},
   journal={Pacific J. Math.},
   volume={138},
   date={1989},
   number={2},
   pages={259-294},
}


\bib{HN1991}{article}{
   author={Hiai, Fumio},
   author={Nakamura, Yoshihiro},
   title={Closed convex hulls of unitary orbits in von Neumann algebras},
   journal={Trans. Amer. Math. Soc.},
   volume={323},
   date={1991},
   number={1},
   pages={1--38},
}



\bib{H1992}{article}{
    author={Holbrook, J. A.},
    title={Spectral variation of normal matrices},
    journal={Linear Algebra Appl.},
    volume={174},
    date={1992},
    pages={131-141}
}



\bib{H1954}{article}{
    author={Horn, A.},
    title={Doubly stochastic matrices and the diagonal of a rotation matrix},
    journal={Amer. J. Math.},
    volume={76},
    number={3},
    date={1954},
    pages={620-630}
}




\bib{HL2015}{article}{
    author={Hu, S.},
    author={Lin, H.},
    title={Distance between unitary orbits of normal elements in simple C$^*$-algebras of real rank zero},
    journal={to appear in J. Funct. Anal.},
    date={2015},
    pages={64 pages}
}







\bib{JST2015}{article}{
    author={Jacelon, B.},
    author={Strung, K. R.},
    author={Toms, A.},
    title={Unitary Orbits of Self Adjoint Operators in Simple $\mathcal{Z}$-Stable C$^*$-Algebras},
    eprint={arXiv:1502.05272},
    date={2015},
    pages={10}
}





\bib{KW2010}{article}{
    author={Kaftal, V.},
    author={Weiss, G.},
    title={An infinite dimensional Schur-Horn theorem and majorization theory},
    journal={J. Funct. Anal.},
    volume={259},
    number={12},
    date={2010},
    pages={3115-3162}
}


\bib{K1983}{article}{
  title={Majorization in finite factors},
  author={Kamei, E.},
  journal={Math. Japon.},
  volume={28},
  number={4},
  pages={495-499},
  date={1983}
}



\bib{K1984}{article}{
  title={Double stochasticity in finite factors},
  author={Kamei, E.},
  journal={Math. Japon.},
  volume={29},
  number={6},
  pages={903-907},
  date={1984}
}



\bib{K1985}{article}{
  title={An order on statistical operators implicitly introduced by von Neumann},
  author={Kamei, E.},
  journal={Math. Japon.},
  volume={30},
  pages={891-895},
  date={1985}
}




\bib{KS2015}{article}{
    author={Kennedy, M.},
    author={Skoufranis, P.},
    title={Thompson's theorem for II$_1$ factors},
    journal={to appear in Trans. Amer. Math. Soc.},
    date={2015},
    pages={23 pages}
}

\bib{L1996}{article}{
    author={Lin, H.},
    title={Approximation by Normal Elements with Finite Spectra in C$^*$-Algebras of Real Rank Zero},
    journal={Pacific J. Math.},
    number={2},
    volume={173},
    date={1996},
    pages={413-487}
}



\bib{MOA2010}{book}{
  author={Marshall, Albert W.},
  author={Olkin, I.},
  author={Arnold, Barry C.},
  title={Inequalities: Theory of majorization and its applications},
  year={2010},
  edition={2},
  series={Series in Statistics},
  publisher={Springer},
}





\bib{MR2014}{article}{
    author={Massey, P.},
    author={Ravichandran, M.},
    title={Multivariable Schur-Horn theorems},
    eprint={arXiv:1411.4457 },
    date={2014},
    pages={31}
}






\bib{M1963}{article}{
    author={Mirsky, L.},
    title={Results and problems in the theory of doubly-stochastic matrices},
    journal={Probab. Theory Related Fields},
    volume={1},
    number={4},
    date={1963},
    pages={319-334}
}



\bib{M1964}{article}{
    author={Mirsky, L.},
    title={Inequalities and existence theorems in the theory of matrices},
    journal={J. Math. Anal. Appl.},
    volume={9},
    number={1},
    date={1964},
    pages={99-118}
}



\bib{M2000}{article}{
    author={Murphy, Gerard},
    title={Uniqueness of the trace and simplicity},
    journal={Proc. Amer. Math. Soc.},
    volume={128},
    number={12},
    date={2000},
    pages={3563-3570}
}




\bib{MN1936}{article}{
    author={Murray, Francis J.},
    author={Neumann, J.},
    title={On rings of operators},
    journal={Ann. of Math.},
    date={1936},
    pages={116-229}
}




\bib{O2013}{article}{
    author={Ozawa, Narutaka},
    title={Dixmier approximation and symmetric amenability for C$^*$-algebras},
    eprint={arXiv:1304.3523},
    date={2013},
    pages={20}
}










\bib{P1985}{article}{
  title={Spectral scale of self-adjoint operators and trace inequalities},
  author={Petz, D{\'e}nes},
  journal={J. Math. Anal. Appl.},
  volume={109},
  number={1},
  pages={74-82},
  date={1985}
}

\bib{P2005}{article}{
  title={Crossed products of the Cantor set by free minimal actions of $\mathbb{Z}^d$},
  author={Phillips, N. Christopher},
  journal={Comm. Math. Phys.},
  volume={256},
  number={1},
  pages={1-42},
  date={2005}
}




\bib{P1980}{article}{
    author={Poon, Y.T.},
    title={Another Proof of a Result of Westwick},
    journal={Linear Algebra Appl.},
    volume={9},
    number={1},
    date={1980},
    pages={35-37}
}



\bib{R2012}{article}{
    author={Ravichandran, M.},
    title={The Schur--Horn theorem in von Neumann algebras},
    eprint={arXiv:1209.0909},
    date={2012},
    pages={22}
}


\bib{R1982}{article}{
    author={Riedel, Norbert},
    title={On the Dixmier property of simple C$^*$-algebras},
    journal={Math. Proc. Camb. Phil. Soc.},
    volume={91},
    date={1982},
    pages={75-78}
}



\bib{R1983}{article}{
    author={Riedel, Norbert},
    title={The weak Dixmier property implies the Dixmier property},
    journal={Dilation theory, Toeplitz operators, and other topics},
    volume={11},
    date={1983},
    pages={299-301}
}







\bib{S1985}{article}{
    author={Sakai, Y.},
    title={Weak spectral order of Hardy, Littlewood and P\'{o}lya},
    journal={J. Math. Anal. Appl.},
    volume={108},
    number={1},
    date={1985},
    pages={31-46}
}






\bib{S1923}{article}{
    author={Schur, I.},
    title={\"{U}ber eine Klasse von Mittelbildungen mit Anwendungen auf die Determinantentheorie},
    journal={Sitzungsber. Berl. Math. Ges.},
    volume={22},
    date={1923},
    pages={9-20}
}




\bib{S2007}{article}{
    author={Sherman, David},
    title={Unitary orbits of normal operators in von Neumann algebras},
    journal={J. Reine Angew. Math.},
    volume={2007},
    number={605},
    date={2007},
    pages={95-132}
}

\bib{S2013-1}{article}{
    author={Skoufranis, P.},
    title={Closed unitary and similarity orbits of normal operators in purely infinite C$^*$-algebras},
    journal={J. Funct. Anal.},
    volume={265},
    number={3},
    year={2013},
    pages={474-506}
}




\bib{ST1992}{article}{
    author={Sunder, V. S.},
    author={Thomsen, K.},
    title={Unitary Orbits of Selfadjoints in some C$^*$-Algebras},
    journal={Houst J Math},
    volume={18},
    number={1},
    date={1992},
    pages={127-137}
}





\bib{T1971}{article}{
    author={Thompson, C. J.},
    title={Inequalities and partial orders on matrix spaces},
    journal={Indiana Univ. Math. J.},
    volume={21},
    number={72},
    date={1971},
    pages={469-480}
}




\bib{T1977}{article}{
    author={Thompson, R. C.},
    title={Singular values, diagonal elements, and convexity},
    journal={SIAM J APPL MATH},
    volume={32},
    number={1},
    date={1977},
    pages={39-63}
}



\bib{W1912}{article}{
    author={Weyl, H.},
    title={Das asymptotische Verteilungsgesetz der Eigenwerte linearer partieller Differentialgleichungen (mit einer Anwendung auf die Theorie der Hohlraumstrahlung)},
    journal={Math. Ann.},
    volume={71},
    number={4},
    date={1912},
    pages={441-479}
}


\bib{Z1990}{article}{
    author={Zhang, S.},
    title={A property of purely infinite simple C$^*$-algebras},
    journal={Proc. Amer. Math. Soc.},
    volume={109},
    number={3},
    date={1990},
    pages={717-720}
}



\end{thebibliography}
\end{document}